\documentclass[letterpaper,11pt]{amsart}
\usepackage[margin=1.2in]{geometry}
\usepackage{amsmath,amsthm,amssymb}
\usepackage{xspace,xcolor}
\usepackage[breaklinks,colorlinks,citecolor=teal,linkcolor=teal,urlcolor=teal,pagebackref,hyperindex]{hyperref}
\usepackage[alphabetic]{amsrefs}
\usepackage[all]{xy}
\usepackage[english]{babel}
\usepackage{enumitem}
\usepackage{tikz,xcolor}
\usepackage{tikz-cd}
\usepackage{mathrsfs}
\usepackage{color}

\usepackage{mathtools}

\newcounter{intro}

\newtheorem{intro-conjecture}[intro]{Conjecture}
\newtheorem{intro-corollary}[intro]{Corollary}
\newtheorem{intro-theorem}[intro]{Theorem}

\newcommand{\theoremref}[1]{\hyperref[#1]{Theorem~\ref*{#1}}}
\newcommand{\sectionref}[1]{\hyperref[#1]{Section~\ref*{#1}}}
\newcommand{\lemmaref}[1]{\hyperref[#1]{Lemma~\ref*{#1}}}
\newcommand{\definitionref}[1]{\hyperref[#1]{Definition~\ref*{#1}}}
\newcommand{\propositionref}[1]{\hyperref[#1]{Proposition~\ref*{#1}}}
\newcommand{\conjectureref}[1]{\hyperref[#1]{Conjecture~\ref*{#1}}}
\newcommand{\corollaryref}[1]{\hyperref[#1]{Corollary~\ref*{#1}}}
\newcommand{\exampleref}[1]{\hyperref[#1]{Example~\ref*{#1}}}
\newcommand{\remarkref}[1]{\hyperref[#1]{Remark~\ref*{#1}}}
\newcommand{\itemref}[1]{\hyperref[#1]{Item~\ref*{#1}}}
\newcommand{\equationref}[1]{\hyperref[#1]{Equation~(\ref*{#1})}}
\newcommand{\formularef}[1]{\hyperref[#1]{Formula~(\ref*{#1})}}
\newcommand{\conditionref}[1]{\hyperref[#1]{Condition~(\ref*{#1})}}

\theoremstyle{plain}
\newtheorem{thm}{Theorem}[section]

\newtheorem{lem}[thm]{Lemma}
\newtheorem{prop}[thm]{Proposition}
\newtheorem{cor}[thm]{Corollary}

\theoremstyle{definition}
\newtheorem{defi}[thm]{Definition}

\newtheorem{eg}[thm]{Example}

\newtheorem{question}[thm]{Question}

\theoremstyle{remark}
\newtheorem{rmk}[thm]{Remark}

\def\Mustata{Mus\-ta\-\c{t}\u{a}\xspace}

\def\Z{{\mathbf Z}}
\def\Q{{\mathbf Q}}

\def\C{{\mathbf C}}

\def\cC{\mathcal{C}}
\def\cD{\mathcal{D}}
\def\cE{\mathcal{E}}

\def\cH{\mathcal{H}}

\def\cM{\mathcal{M}}
\def\cN{\mathcal{N}}
\def\cO{\mathcal{O}}

\def\cT{\mathcal{T}}

\def\.{\cdot}
\def\^{\widehat}

\def\de{\partial}

\def\({\left(}
\def\){\right)}

\renewcommand{\and}{ \ \ \text{ and } \ \ }

\begin{document}

\author[Q.~Chen]{Qianyu Chen}

\address{Department of Mathematics, University of Michigan, 530 Church Street, Ann Arbor, MI 48109, USA}
\address{Institute of Geometry and Physics, University of Science and Technology of China, 96
Jinzhai Road, Baohe District, Hefei, China}

\email{qyc@umich.edu, qianyu.chen16@gmail.com}

\author[B.~Dirks]{Bradley Dirks}

\address{Department of Mathematics, Stony Brook University, Stony Brook, NY 11794-3651, USA}

\email{bradley.dirks@stonybrook.edu}

\author[S.~Olano]{Sebasti\'{a}n Olano}

\address{Department of Mathematics, University of Toronto, 40 St. George St., Toronto, Ontario Canada, M5S 2E4}

\email{seolano@math.toronto.edu}

\thanks{Q.C. was partially supported by NSF grant DMS-1952399 and an AMS-Simons travel grant. B.D. was partially supported by the National Science Foundation under Grant No. DMS-1926686 and MSPRF DMS-2303070.}

\title[Filtrations on Local Cohomology and Injectivity Theorems]{Filtrations on Local Cohomology and Injectivity Theorems}

\begin{abstract} We prove that the Hodge filtration on local cohomology is contained in the symbolic Ext filtration, and hence in the usual Ext filtration, answering a question of Mustaţă and Popa. Our main observation is that the Ext filtration admits a filtered-duality interpretation via Hartshorne's algebraic de Rham complex and its infinitesimal filtration. Comparing this construction with the filtered Du Bois complex yields the containment. The same mechanism proves the higher injectivity conjecture of Popa--Shen--Vo and several variants.
\end{abstract}

\maketitle
\section{Introduction}
Let $X$ be a closed subvariety of a smooth complex algebraic variety $Y$. Saito's theory of mixed Hodge modules endows the local cohomology modules $\cH^q_X(\cO_Y)$ with an exhaustive \emph{Hodge} filtration $F_\bullet \cH^q_X(\cO_Y)$ by $\cO_Y$-coherent submodules. On the other hand, the local cohomology sheaf also has the direct-limit presentation:
\[ \cH^q_X(\cO_Y) = \varinjlim \cE xt^q_{\cO_Y}(\cO_Y/I^{p+1},\cO_Y) = \varinjlim \cE xt^q_{\cO_Y}(\cO_Y/I^{(p+1)},\cO_Y),\]
where $I \subseteq \cO_Y$ is the ideal sheaf defining $X$ in $Y$. Here $I^{(p+1)}$ denotes its $(p+1)$-st symbolic power, and both direct limits are taken over $p \geq 0$. Following \cite{MP3}, the images of the corresponding Ext sheaves define the Ext filtrations $E_\bullet^{\rm symb} \cH^q_X(\cO_Y) \subseteq E_\bullet \cH^q_X(\cO_Y)$, which are also exhaustive filtrations by $\cO_Y$-coherent submodules. 

The Hodge filtration is transcendental in nature, while the Ext filtration is algebraic. Motivated by the hypersurface case, \Mustata and Popa \cite{MP3}*{Question 3.5} asked whether we have containment $F_p \cH^q_X(\cO_Y) \subseteq E_p \cH^q_X(\cO_Y)$ for all $p$ and $q$ in $\Z_{\geq 0}$. This containment follows from Saito's construction when $q = \dim Y - \dim X$, as observed by \Mustata and Popa. Moreover, for $p = 0$, they showed by birational methods that the containment holds.

We fully answer this question by proving a stronger result. We show not only that the Hodge filtration is contained in the Ext filtration, but that each Hodge piece factors through the corresponding (symbolic) Ext sheaf.

\begin{intro-theorem} \label{thm-HodgeInExt} Let $X$ be a closed subvariety of a smooth variety $Y$, defined by the ideal sheaf $I\subseteq \cO_Y$. Then for all $p,q \in \Z_{\geq 0}$, we have the commutative diagram
\[ \begin{tikzcd} {} & \cE xt^q_{\cO_Y}(\cO_Y/I^{(p+1)},\cO_Y) \ar[r] \ar[d] & \cE xt^q_{\cO_Y}(\cO_Y/I^{p+1},\cO_Y) \ar[dl]\\ F_p \cH^q_X(\cO_Y) \ar[r,hook] \ar[ur,hook] & \cH^q_X(\cO_Y) & {} \end{tikzcd},\]
and so we have containments
\[ F_\bullet \cH^q_X(\cO_Y) \subseteq E^{\rm symb}_\bullet \cH^q_X(\cO_Y) \subseteq E_\bullet \cH^q_X(\cO_Y).\]

Moreover, the (symbolic) Ext filtration on $\cH^q_X(\cO_Y)$ is compatible with $(\cD_Y,F_\bullet)$.
\end{intro-theorem}

For divisors, the Ext filtration agrees with the pole order filtration. The containment of the Hodge filtration in the pole order filtration has a long history dating back to Griffiths \cite{GriffithsPole} and Deligne--Dimca \cite{DeligneDimca}. 

As explained in \cite{MP3}, the injectivity of the leftmost diagonal arrow in \theoremref{thm-HodgeInExt}, for $p=0$, is equivalent to the Kov\'acs--Schwede injectivity theorem \cite{DBDeform}*{Thm. 3.3}. Interestingly, the proof of this case in \cite{MP3} uses strictness of direct images in Saito's theory of mixed Hodge modules, while the original proof in \cite{DBDeform} uses the $E_1$-degeneration of the Du Bois--de Rham spectral sequence (another incarnation of strictness for proper pushforwards). Our proof also uses a strictness result, but for duality rather than proper pushforwards.

Despite their different definitions, the Hodge and Ext filtrations admit a common interpretation through filtered duality. On the Hodge-theoretic side, the filtered Du Bois complex encodes the Hodge filtration. On the algebraic side, Hartshorne's algebraic de Rham complex \cite{HartshorneDR} carries an infinitesimal filtration coming from the infinitesimal neighborhoods of $X$ \cite{BhattDDR}. Filtered duality provides a bridge between these two descriptions.

More precisely, Bhatt observed that Hartshorne's algebraic de Rham complex, equipped with its infinitesimal filtration, admits a natural filtered morphism to the Du Bois complex \cite{BhattDDR}*{Prop. 5.2}. We also introduce the symbolic infinitesimal filtration and factor Bhatt's morphism through it:
\begin{equation} \label{eq-mainDiagram} (\Omega_{X/\C}^H, {\rm Fil}_{\rm inf}^\bullet) \to (\Omega_{X/\C}^H,{\rm Fil}_{\rm symbinf}^\bullet) \to (\underline{\Omega}_X^\bullet,F^\bullet).\end{equation}

After applying filtered duality, the rightmost object recovers local cohomology with its Hodge filtration, while the first two complexes recover the ordinary and symbolic Ext approximations, respectively. Thus the filtered morphisms above give rise to the comparison maps appearing in \theoremref{thm-HodgeInExt}.

The following result makes the algebraic side of this principle precise. Applying filtered duality to the infinitesimally filtered algebraic de Rham complex gives a complex of filtered $\cD_Y$-modules, which we denote by $(\cE^\bullet(X),F)$. Its underlying complex computes local cohomology, while its filtration records the Ext sheaves.

\begin{intro-theorem} \label{thm-InterpretExt} There is a quasi-isomorphism $\cE^\bullet(X) \simeq R\Gamma_X(\cO_Y)$ in $D^b_{\rm hol}(\cD_Y)$.

Moreover, we have the identifications
\[ \cH^q F_p \cE^\bullet(X) \cong \cE xt^{q}(\cO_Y/I^{p+1},\cO_Y),\]
\[ F_p \cH^q \cE^\bullet(X) \cong E_{p} \cH^q_X(\cO_Y).\]

Hence, the condition 
\[ \cE xt^q_{\cO_Y}(\cO_Y/I^{p+1},\cO_Y) \to \cH^q_X(\cO_Y) \text{ is injective for all } p \leq m \text{ and all } q\in \Z\]
is equivalent to $(\cE^\bullet(X),F_{\min\{\bullet,m\}})$ being a strict filtered complex.
\end{intro-theorem}

Eisenbud--\Mustata--Stillman \cite{EMS}*{Question 6.2} asked for which ideals the injectivity condition of the theorem holds. Thus, the theorem gives a relation between this injectivity problem and strictness for a complex of filtered $\cD_Y$-modules.

\begin{rmk} The first statement in the above theorem is an incarnation of Hartshorne's result \cite{HartshorneDR}*{IV.1.1} showing that the complex $\Omega_{X/\C}^H$, after analytification, is a resolution of the constant sheaf $\C_{X^{\rm an}}$. However, we emphasize that Hartshorne's result is not an input to our proof.
\end{rmk}

The same filtered-duality mechanism that proves \theoremref{thm-HodgeInExt} also gives the following higher injectivity results. In fact, it proves the stronger result, \theoremref{thm-Injectivity2}, which also involves new classes of singularities introduced in \cite{CDOHigher}. The cotangent complex $L_{X/\C}$ arises by applying our method to Bhatt's derived de Rham complex $\widehat{\rm dR}_X$, with its derived Hodge filtration. This filtered differential complex can be placed on the far left side of the diagram \eqref{eq-mainDiagram}, by \cite{BhattDDR}*{Prop. 5.2}.


For $m\in \Z_{\geq -1}$, recall \cite{SVV} that $X$ is \emph{pre-$m$-Du Bois} if the natural map
\[ \cH^0(\underline{\Omega}_X^p) \to \underline{\Omega}_X^p\]
is a quasi-isomorphism for all $p\leq m$, and it is \emph{strict-$m$-Du Bois} if the natural map
\[\Omega_X^p \to \underline{\Omega}_X^p\]
is a quasi-isomorphism for all $p\leq m$. For $m = -1$, both conditions are vacuously true.

\begin{intro-theorem} \label{thm-Injectivity} Let $X$ be a complex algebraic variety and let $k \in \Z_{\geq 0}$.

\begin{enumerate} \item If $X$ is pre-$(k-1)$-Du Bois, then the natural morphism
\[ \cE xt_{\cO_X}^q(\underline{\Omega}_X^k,\omega_X^\bullet) \to \cE xt_{\cO_X}^q(\cH^0(\underline{\Omega}_X^k),\omega_X^\bullet)\]
is injective for all $q\in \Z$.
\item If $X$ is strictly $(k-1)$-Du Bois, then the natural morphism
\[ \cE xt_{\cO_X}^q(\underline{\Omega}_X^k,\omega_X^\bullet) \to \cE xt_{\cO_X}^q(\Omega_X^k,\omega_X^\bullet)\]
is injective for all $q\in \Z$.





\item If the natural morphism
\[ \bigwedge^p L_{X/\C} \to \underline{\Omega}_X^p\]
is a quasi-isomorphism for all $p < k$, then the natural morphism
\[ \cE xt^q_{\cO_X}(\underline{\Omega}_X^k,\omega_X^\bullet) \to \cE xt^q_{\cO_X}(\bigwedge^k L_{X/\C}, \omega_X^\bullet)\]
is injective for all $q\in \Z$.
\end{enumerate}
\end{intro-theorem}

The first statement was proved in \cite{PSV}*{Theorem D} for isolated singularities and conjectured to hold in general. The non-isolated case of the first statement was established by Kov\'acs \cite{KovacsInjectivity}*{Theorem 1.1} using a different method, closer in spirit to the original proof of the $p=0$ case. The second statement was proved for varieties with LCI singularities in \cite{MPDB}*{Theorem A}. For $k > 1$, the assumption in the last statement implies, in particular, that $L_{X/\C} \simeq \Omega_X^1$, which is quite restrictive. However, for $k = 1$, we have the useful restatement, with interpretations related to the local deformation theory of $X$.

\begin{intro-corollary} \label{cor-InjectiveCotangent} If $X$ has Du Bois singularities, then the natural morphism
\[ \cE xt^q_{\cO_X}(\underline{\Omega}_X^1,\omega_X^\bullet) \to \cE xt^q_{\cO_X}(L_{X/\C},\omega_X^\bullet)\]
is injective for all $q\in \Z$.

In particular, if $X$ has Gorenstein Du Bois singularities, we get a canonical inclusion
\[ \cE xt^1_{\cO_X}(\underline{\Omega}_X^1, \cO_X) \hookrightarrow \cT^1_X,\]
where the right-hand side is the local first-order deformation sheaf.
\end{intro-corollary}

If $X$ has isolated singularities, and if $f \colon \widetilde{X} \to X$ is a resolution of singularities which is an isomorphism over $X\setminus X_{\rm sing}$, then the Du Bois complex admits a description in terms of the resolution by \cite{PSMHS}*{Example 7.25}. So we get the following version of the last statement in \corollaryref{cor-InjectiveCotangent}.

\begin{intro-corollary} \label{cor-InjectiveCotangentBirational} Assume $X$ has isolated, Gorenstein, Du Bois singularities with $f\colon \widetilde{X} \to X$ a resolution as above. Suppose $E = f^{-1}(X_{\rm sing})_{\rm red}$ has simple normal crossings. 
Then we have a natural inclusion for every $x\in X_{\rm sing}$:
\[ R^1 f_* (T_{\widetilde{X}} (-\log E)(E) \otimes_{\cO_{\widetilde{X}}} \omega_{\widetilde{X}/X})_x \hookrightarrow T^1_{X,x}.\]
\end{intro-corollary}

\noindent {\bf Acknowledgments.} We thank Bhargav Bhatt, S\'{a}ndor Kov\'{a}cs, Mircea \Mustata, Sung Gi Park, Mihnea Popa, Christian Schnell and Wanchun Shen for helpful discussions.

\section{Background on Local Cohomology and de Rham complexes} 
\noindent {\bf Notation and Conventions.} We collect here important notation from Section \ref{sec-FiltDiff} which is used throughout the paper.

\begin{enumerate} 
\item  
For operations related to filtered $\cD$-modules, we use \emph{increasing} filtrations with subscript notation. 
For Bhatt's constructions, we retain the decreasing filtrations with superscript notation, but we re-index when using Saito's formalism: the indexing conventions are related throughout this paper by
\[ F_p = F^{-p}.\]

\item In the introduction, we preferred to use left filtered $\cD$-modules rather than right filtered $\cD$-modules for readability. However, in almost all proofs below, we will stick to the right $\cD$-module category. They are related in the following way: for a left $\cD_Y$-module $\cM$, the corresponding right $\cD_Y$-module is $\cM^r = \omega_Y \otimes_{\cO_Y} \cM$ with filtration
\[ F_{\bullet-\dim Y} \cM^r = \omega_Y\otimes_{\cO_Y}(F_\bullet \cM).\]

\item $\mathbf D(-)$, duality for mixed Hodge modules (discussed in \propositionref{prop-strictnessDual});

\item $\mathbb D(-)$, duality for filtered differential complexes (see \definitionref{def-dualityfdc});

\item $\mathbb D^{\rm coh}_X(-) \coloneqq R\cH om_{\cO_X}(-,\omega_X^\bullet)$, Grothendieck duality for $\cO_X$-modules, where $\omega_X^\bullet$ is a normalized  dualizing complex.

\item $\widetilde{\rm DR}(-)$, the filtered differential complex associated with a filtered $\cD$-module (see \remarkref{rmk-LeftDMods});

\item ${\rm DR}^{-1}(-)$, the induced filtered $\cD$-module associated with a filtered differential complex (see \definitionref{defi-DRInverse}).
\end{enumerate}

\subsection{Set-up} \label{subsec-Setup} Throughout, $X$ denotes a complex algebraic variety (not necessarily irreducible).

Locally, fix a closed embedding $\iota \colon X \hookrightarrow Y$ with $Y$ a smooth connected variety. Let $I \subseteq \cO_Y$ denote the coherent sheaf of ideals defining $X$ and denote $c = {\rm codim}_Y(X)$.

\subsection{Ext and Local Cohomology}
We work locally, using the notation from Subsection \ref{subsec-Setup}.

The projection $\cO_Y/I^j \to \cO_Y/I^k$ for any $j \geq k \geq 0$ gives a direct system for any $q\in \Z$:
\[ \cE xt^q_{\cO_Y}(\cO_Y/I^k, \omega_Y) \to \cE xt^q_{\cO_Y}(\cO_Y/I^j,\omega_Y),\]
the direct limit of which is the \emph{local cohomology module} $\cH^q_X(\omega_Y)$. By definition, we have morphisms
\[ \psi^q_k \colon \cE xt^q_{\cO_Y}(\cO_Y/I^k,\omega_Y) \to \cH^q_X(\omega_Y).\]

Following \cite{MP3}, we define the \emph{Ext filtration} on local cohomology by
\[ E_{k-\dim Y} \cH^q_X(\omega_Y) = {\rm im}(\psi^q_{k+1}),\]
where the choice of indexing is made to match that of the Hodge filtration. When working with left $\cD$-modules, we define
\[ E_p \cH^q_X(\cO_Y) = {\rm im}(\cE xt^q_{\cO_Y}(\cO_Y/I^{p+1},\cO_Y) \to \cH^q_X(\cO_Y)),\]
which is again chosen to agree with the indexing conventions for the Hodge filtration.

\begin{rmk} \label{rmk-ExtIsOrder} As shown in \cite{MP3}*{Sec. 7}, for $q = c$, the Ext filtration agrees with the \emph{order filtration}
\[ E_{k-\dim Y} \cH^c_X(\omega_Y) = \{ m \in \cH^c_X(\omega_Y) \mid I^{k+1} \cdot m = 0\},\]
and this interpretation will be quite useful to us in our later arguments.

Moreover, once again for $q=c$, the morphism $\psi^c_{k+1}$ is always injective. Thus, we have an identification
\[ \cE xt^c_{\cO_Y}(\cO_Y/I^{k+1},\omega_Y) \cong E_{k-\dim Y} \cH^c_X(\omega_Y).\]
\end{rmk}

Recall that for any $j\geq 0$ the ordinary powers $I^j$ are always contained in the \emph{symbolic powers} $I^j \subseteq I^{(j)}$. For more details on symbolic powers, we refer the reader to \cite{Eisenbud}*{Sec. 3.9}.

As $Y$ is smooth, the families $\{I^j\}$ and $\{I^{(j)}\}$ are cofinal \cite{HH}*{Thm. 1.1(a)}. In particular, we also have an isomorphism
\[ \varinjlim_j \cE xt^q_{\cO_Y}(\cO_Y/I^{(j)},\omega_Y) \cong \cH^q_X(\omega_Y),\]
and can define, analogously to the filtration $E_\bullet \cH^q_X(\omega_Y)$, the \emph{symbolic Ext filtration} $E_\bullet^{\rm symb} \cH^q_X(\omega_Y)$. The containment $I^{p+1} \subseteq I^{(p+1)}$ gives a factorization
\[ \cE xt^q_{\cO_Y}(\cO_Y/I^{(p+1)},\omega_Y) \to \cE xt^q_{\cO_Y}(\cO_Y/I^{p+1},\omega_Y) \to \cH^q_X(\omega_Y),\]
and so the symbolic Ext filtration is finer than the usual Ext filtration:
\[ E_\bullet^{\rm symb} \cH^q_X(\omega_Y) \subseteq E_\bullet \cH^q_X(\omega_Y).\]

\subsection{Algebraic de Rham complex}
Hartshorne \cite{HartshorneDR} gave an algebraic construction of a complex which computes the Betti cohomology of the associated complex analytic variety. For singular $X$, the usual de Rham complex $\Omega_X^\bullet$ (formed by $\Omega_X^1$ and its exterior powers) need not compute this cohomology.

We work in the local set-up of Subsection \ref{subsec-Setup}; however, these notions glue on $X$ \cite{HartshorneDR}. The main observation is that the completion $\widehat{\cO}_Y = \varprojlim \cO_Y/I^j$ along $X$ admits a $\cD_Y$-module structure. Indeed, let $\alpha = (\alpha_j)_{j \geq 1} \in \widehat{\cO}_Y$ and choose local lifts $g_j\in \cO_Y$ of $\alpha_j  \in \cO_Y/I^j$. Then we define, for any $\theta \in \cT_Y$, the element $\theta \alpha \in \widehat{\cO}_Y$ defined by
\[ (\theta \alpha)_j = {\theta(g_{j+1})} \bmod I^j.\]

This is clearly another compatible system of representatives: by definition, $g_{j+1} - g_{k+1} \in I^{\min\{j+1,k+1\}}$, and so $\theta(g_{j+1}) - \theta(g_{k+1}) \in I^{\min\{j+1,k+1\}-1} = I^{\min\{j,k\}}$. A similar computation shows it is well-defined. Moreover, if one instead views $\widehat{\cO}_Y = \varprojlim \cO_Y/I^{(j)}$, then this computation shows that one obtains the same $\cD_Y$-module structure.

The \emph{algebraic de Rham complex} of $X$ (over $\C$) is, by definition, the de Rham complex for $\widehat{\cO}_Y$ with this $\cD_Y$-module structure, and is denoted $\Omega_{X/\C}^H$. Conventionally, it is placed in degrees $0,\dots, \dim Y$, rather than $-\dim Y,\dots, 0$ (the latter being the convention in usual $\cD$-module theory).

Bhatt \cite{BhattDDR} defines the \emph{infinitesimal filtration} ${\rm Fil}^\bullet_{\rm inf} \Omega_{X/\C}^H$ as the filtration induced by the $I$-adic filtration on $\widehat{\cO}_Y$, and explains \cite{BhattDDR}*{Rmk. 4.8} (see also \cite{BdJ}*{Rmk. 3.7}) that this is independent of the choice of smooth embedding. 

Explicitly, one has
\[ \Omega_{X/\C}^H/ {\rm Fil}_{\rm inf}^p = \left[ \widehat{\cO}_Y/I^p\widehat{\cO}_Y \xrightarrow[]{\nabla} \widehat{\cO}_Y/I^{p-1}\widehat{\cO}_Y \otimes_{\cO_Y} \Omega_Y^1 \xrightarrow[]{\nabla} \widehat{\cO}_Y/I^{p-2}\widehat{\cO}_Y \otimes_{\cO_Y} \Omega_Y^2 \to \cdots \right]\]
\[ \cong \left[ \cO_Y/I^p \xrightarrow[]{\nabla} \cO_Y/I^{p-1} \otimes_{\cO_Y} \Omega_Y^1 \xrightarrow[]{\nabla} \cO_Y/I^{p-2} \otimes_{\cO_Y} \Omega_Y^2 \to \cdots \right],\]
with the convention that $I^p=\cO_Y$ for $p\leq 0$.


Define the \emph{symbolic infinitesimal filtration} ${\rm Fil}_{\rm symbinf}^\bullet$ on $\Omega^H_{X/\C}$ analogously, using the filtration by symbolic powers of $I$ on $\widehat{\cO}_Y$. The containment $I^p \subseteq I^{(p)}$ gives a filtered morphism
\[ (\Omega_{X/\C}^H,{\rm Fil}^\bullet_{\rm inf}) \to(\Omega_{X/\C}^H,{\rm Fil}^\bullet_{\rm symbinf}).\]

\begin{rmk} \label{rmk-MapDuBois} Bhatt \cite{BhattDDR}*{Prop. 5.2} observes that the algebraic de Rham complex, with the infinitesimal filtration, has a natural filtered morphism to $(\Omega_X^\bullet,F)$, the usual complex of K\"{a}hler differentials on $X$. Hence, composing with the natural filtered morphism
\[ (\Omega_X^\bullet,F) \to (\underline{\Omega}_X^\bullet,F),\]
we get a filtered morphism $(\Omega_{X/\C}^H, {\rm Fil}_{\rm inf}^\bullet) \to (\underline{\Omega}_X^\bullet,F)$.

This morphism factors through the symbolic infinitesimal filtration. Indeed, we have the morphisms
\[ \Omega_{X/\C}^H/{\rm Fil}_{\rm inf}^p \to \Omega_{X/\C}^H/{\rm Fil}_{\rm symbinf}^p \to \Omega_X^\bullet/F^p\]
given by
\[ 
\begin{tikzcd} 
 \cO_Y/I^p \ar[r] \ar[d] & \cO_Y/I^{p-1} \otimes_{\cO_Y} \Omega_Y^1 \ar[r] \ar[d] & \cdots \ar[r] & \cO_Y/I \otimes_{\cO_Y} \Omega_Y^{p-1} \ar[d]\ar[r] & 0 \ar[d] \\ 
 \cO_Y/I^{(p)} \ar[r] \ar[d] & \cO_Y/I^{(p-1)} \otimes_{\cO_Y} \Omega_Y^1 \ar[r] \ar[d] & \cdots \ar[r] & \cO_Y/I \otimes_{\cO_Y} \Omega_Y^{p-1} \ar[d] \ar[r] & 0 \ar[d] \\
 \cO_X \ar[r]& \Omega_X^1 \ar[r]  & \cdots \ar[r] & \Omega_X^{p-1} \ar[r] & 0 \end{tikzcd},
\]
where the vertical morphisms from the first row are the obvious surjections and the surjections from the second row to the third factor 
\[ \cO_Y/I^{(p-j)} \otimes_{\cO_Y} \Omega_Y^j \to \cO_Y/I \otimes_{\cO_Y} \Omega_Y^j \to \Omega_X^j.\]

In summary, we have filtered morphisms
\[ (\Omega_{X/\C}^H,{\rm Fil}_{\rm inf}^\bullet) \to (\Omega_{X/\C}^H,{\rm Fil}_{\rm symbinf}^\bullet) \to (\Omega_X^\bullet,F^\bullet) \to (\underline{\Omega}_X^\bullet,F^\bullet).\]
\end{rmk}

\begin{rmk} \label{rmk-computeGrInf} Immediately from the definition, we have isomorphisms
\[ {\rm Gr}_{\rm inf}^p \Omega_{X/\C}^H \cong \left[ I^p/I^{p+1} \to I^{p-1}/I^{p} \otimes_{\cO_Y} \Omega_Y^1 \to \cdots \right],\]
\[ {\rm Gr}_{\rm symbinf}^p \Omega_{X/\C}^H \cong \left[ I^{(p)}/I^{(p+1)} \to I^{(p-1)}/I^{(p)} \otimes_{\cO_Y} \Omega_Y^1 \to \cdots \right].\]

As in Remark \ref{rmk-MapDuBois}, it is easy to check that, for $p \leq \dim Y$, the rightmost cohomology of either complex is $\Omega_X^p$.
\end{rmk}

It is well known that the underlying complex $\Omega_{X/\C}^H$ computes the Betti cohomology for $X$ (this is the main point of \cite{HartshorneDR}*{Ch. 4}). If $X$ is projective, Bhatt in~\cite{BhattDDR} uses this filtered complex to define a containment
\[ F^\bullet_{\rm inf} H^k(X^{an}) \subseteq F^\bullet H^k(X^{an}),\]
where the second term is the Hodge filtration from Deligne's theory.

\subsection{Derived de Rham Complex} We recall Bhatt's construction of the \emph{Hodge-completed derived de Rham complex} of a $\C$-variety $X$. We explain the construction in the affine case $X = {\rm Spec}(R)$. Fix $P_\bullet \to R$ a simplicial polynomial resolution such that each $P_n$ is of finite type over $\C$. The constructions are, up to homotopy, independent of the choices and glue for non-affine $X$. Here $Y = {\rm Spec}(P_0)$ will play the role of our ambient smooth variety.

For any $p\in \Z$, let $L \Omega_{R/\C}^{\leq p}$ denote the total complex of the double complex $\Omega^{\leq p}_{P_\bullet/\C}$. Then define
\[ \widehat{\rm dR}_X = \varprojlim_p L\Omega_{R/\C}^{\leq p},\]
with Hodge filtration ${\rm Fil}_H^\bullet$ defined by
\[ \widehat{\rm dR}_X/ {\rm Fil}_H^{p+1} = L\Omega_{R/\C}^{\leq p}.\]

This satisfies
\[ {\rm Gr}_H^p \widehat{\rm dR}_X \simeq \bigwedge^p L_{X/\C}[-p],\]
where $L_{X/\C}$ is the cotangent complex of $X\to {\rm Spec}(\C)$, and the wedge power is the \emph{derived} wedge power. By construction, $\widehat{\rm dR}_X$ carries a dg-algebra (DGA) structure.

\begin{rmk}\label{rmk-FunctorialSmoothEmbed} By \cite{BhattDDR}*{Rmk. 4.3}, this construction is functorial. In particular, given $X\hookrightarrow Y$ a closed embedding into a smooth variety $Y$, we have an induced filtered morphism
\[ (\widehat{\rm dR}_Y,{\rm Fil}_H^\bullet) \to (\widehat{\rm dR}_X,{\rm Fil}_H^\bullet),\]
and by \cite{BhattDDR}*{Ex. 4.4}, the left object can be canonically identified with $(\Omega_Y^\bullet, F^\bullet)$.
\end{rmk}

Set $S = P_0$, with $I \subseteq S$ the defining ideal of $R$, so that $X$ is a closed subvariety of the smooth variety $Y = {\rm Spec}(S)$.

Explicitly, write
\[ (\Omega^{\leq p}_{P_\bullet/\C})^j = \bigoplus_{ s \leq p, s -\ell = j} \Omega^s_{P_\ell/\C},\]
where the simplicial degree $\ell$ contributes a cohomological degree of $-\ell$, as simplicial resolutions are often homologically indexed: $\cdots \to P_2 \to P_1 \to P_0$. 

The total differential combines the simplicial and de Rham differentials:
\[ d \vert_{\Omega^s_{P_\ell/\C}} = \sum_{i=0}^\ell (-1)^i d_i^* + (-1)^\ell d_{\rm DR},\]
where $d_i \colon P_\ell \to P_{\ell-1}$ is the $i$-th face map of the simplicial object $P_\bullet$.

Passing to the inverse limit gives
\[ \widehat{\rm dR}_X^j = \prod_{s-\ell = j} \Omega_{P_\ell/\C}^s,\]
with the surjection $\widehat{\rm dR}_X^j \to \widehat{\rm dR}_X^j/{\rm Fil}_H^{p+1}$ being given by the natural morphism
\[ \prod_{s-\ell = j} \Omega_{P_\ell/\C}^s \to \bigoplus_{ s \leq p, s - \ell = j} \Omega^s_{P_\ell/\C} .\]

As Bhatt explains \cite{BhattDDR}*{Rmk. 4.2}, $\widehat{\rm dR}_X$ is an $E_\infty$-algebra. Using our explicit representative, we can make the DGA structure precise: it is given by the shuffle product in the simplicial direction and the wedge product structure on $\Omega^\bullet_{P_\ell/\C}$, with the usual Koszul signs. With this structure, the natural morphism
\[ \Omega_{P_0/\C}^\bullet \to \widehat{\rm dR}_X\]
given by including the simplicial degree $0$ part is a morphism of DGAs. This can also be seen from the point of view of functoriality as in Remark \ref{rmk-FunctorialSmoothEmbed}.

By \propositionref{prop-Construction2}, this DGA structure and the morphism from $\Omega_{P_0/\C}^\bullet$ show that $(\widehat{\rm dR}_X, {\rm Fil}^H_\bullet)$ is a \emph{filtered differential complex}, to which we now turn. We switched to increasing filtrations to match the conventions for filtered $\cD$-modules.

\section{Filtered Duality and the Ext Filtration} \label{sec-FiltDiff}
The purpose of this section is to develop the filtered-duality mechanism underlying the main results \theoremref{thm-HodgeInExt} and \theoremref{thm-Injectivity}. We first briefly recall Saito's filtered differential complexes. Then we extend the relevant duality functor to complexes that also include the algebraic and derived de Rham complexes. The key computation in Section \ref{ss-DualityExt} relates the algebraic de Rham complexes to the Ext filtration. Section \ref{ss-HodgeInput} contains the Hodge-theoretic duality inputs from Saito's theory.

\subsection{Filtered Differential Complexes}
Many applications of mixed Hodge modules to algebraic geometry use properties of the graded de Rham functor ${\rm Gr}^F_k {\rm DR}(-)$: for example, it is compatible with proper pushforward and duality. These properties are explained in depth in \cite{SaitoMHP}*{Sec. 2}.

Saito's theory of filtered differential complexes is central to these questions. Let $Y$ be a smooth complex algebraic variety. Given a right $\cD_Y$-module $\cM$, the filtered Spencer complex (which, from now on, we will call the \emph{filtered de Rham complex}, following Saito)
\[ 
\begin{aligned}
    &{\rm DR}(\cM,F) = \\
    &\left[(\cM,F[\dim Y]) \otimes_{\cO_Y} \bigwedge^{\dim Y} \cT_Y \xrightarrow[]{\nabla} (\cM,F[\dim Y-1])  \otimes_{\cO_Y} \bigwedge^{\dim Y- 1} \cT_Y \xrightarrow[]{\nabla} \dots \xrightarrow[]{\nabla} (\cM,F) \right]    
\end{aligned}
\]
does not have filtered $\cO$-linear differentials. Rather, the differentials are \emph{filtered differential morphisms}, in the following sense.

\begin{defi}[Filtered Differential Morphism] 
Let $(L_1,F), (L_2,F)$ be two filtered $\cO_Y$-modules. Then $(\cM_i,F) = ( L_i,F) \otimes_{\cO_Y} (\cD_Y,F)$ is a filtered right $\cD_Y$-module for $i = 1,2$, where $F_\bullet \cM_i$ is the convolution of the filtration $F_\bullet L_i$ and the order filtration $F_\bullet \cD_Y$.

We have a natural morphism
\[
\begin{aligned}
    {\rm DR} \colon \quad {\rm Hom}_{\cD_Y}(\cM_1 ,\cM_2) &\to \quad {\rm Hom}_{\C}(L_1,L_2) \\
    \psi \quad  &\mapsto \left(L_1 \to L_1\otimes 1 \to \cM_1 \xrightarrow[]{\psi} \cM_2 \to \cM_2 \otimes_{\cD_Y}\cO_Y = L_2 \right)
\end{aligned}
\]
which is injective \cite{SaitoMHP}*{Lem. 2.2.2}.

The image of the composition
\[{\rm Hom}_{\cD_Y}((\cM_1,F) ,(\cM_2,F)) \subseteq {\rm Hom}_{\cD_Y}(\cM_1,\cM_2) \xrightarrow[]{\rm DR} {\rm Hom}_{\C}(L_1,L_2)\]
is the set of \emph{filtered differential morphisms} ${\rm Hom}_{\rm Diff}((L_1,F),(L_2,F))$.

Following Saito, when we view ${\rm DR}(\cM,F)$ as a filtered differential complex, we prefer to write it as $\widetilde{\rm DR}(\cM,F)$.
\end{defi}

\begin{rmk} \label{rmk-LeftDMods} If $(\cM,F)$ is a filtered \emph{left} $\cD_Y$-module, then the filtered de Rham complex
\[ 
\begin{aligned}
    &\widetilde{\rm DR}(\cM,F) = \\
    &   \left[(\cM,F) \xrightarrow[]{\nabla} \Omega_Y^1 \otimes_{\cO_Y} (\cM,F[-1]) \xrightarrow[]{\nabla} \cdots \xrightarrow[]{\nabla} \omega_Y \otimes_{\cO_Y} (\cM,F[-\dim Y])  \right] 
\end{aligned}
\]
is also a filtered differential complex (compatibly with the filtered side-changing functor).
\end{rmk}

\begin{rmk} \label{rmk-FiltOLinear} If $\varphi \colon (L_1,F)\to (L_2,F)$ is an $\cO$-linear filtered morphism, then we get a filtered $\cD_Y$-linear morphism $\varphi \otimes {\rm id}_{\cD_Y} \colon \cM_1 \to \cM_2$ which clearly maps to $\varphi$ under ${\rm DR}$. Thus, any filtered $\cO_Y$-linear morphism is automatically a filtered differential morphism.
\end{rmk}

\begin{rmk} Saito gives \cite{SaitoMHP}*{(2.2.4.1), Lem. 2.2.5} an alternative characterization of filtered differential morphisms from which it is easy to see that if $\varphi \in {\rm Hom}_{\rm Diff}((L_1,F),(L_2,F))$, then $\varphi$ is filtered, and the induced morphism ${\rm Gr}^F_p \varphi \colon {\rm Gr}^F_p L_1 \to {\rm Gr}^F_p L_2$ is $\cO_Y$-linear for all $p\in \Z$.
\end{rmk}

A filtered $\cD_Y$-module of the form $(\cM,F) = (L,F) \otimes_{\cO_Y} (\cD_Y,F)$ is called an \emph{induced filtered $\cD_Y$-module}. Induced $\cD$-modules play a key role in defining the duality and pushforward functors for filtered $\cD$-modules \cite{SaitoMHP}*{Sec. 2}. Under the standard bounded-below hypothesis for filtrations of $\cD$-modules, every module admits a canonical resolution by induced filtered $\cD_Y$-modules:
\begin{lem}[\cite{SaitoMHP}*{Lem. 2.1.6}] \label{lem-InducedResolution} Let $(\cM,F)$ be a filtered $\cD_Y$-module such that $F_p \cM = 0$ for $p \ll 0$. There is a canonical resolution
\[ \left((\cM,F[-\bullet]) \otimes_{\cO_Y} \bigwedge^{-\bullet} \cT_Y\right) \otimes_{\cO_Y}(\cD_Y,F) \to (\cM,F)\]
by induced filtered $\cD_Y$-modules. 
\end{lem}

\begin{defi}[De Rham Inverse Functor, \cite{SaitoMHP}*{Cor. 2.2.6}] \label{defi-DRInverse} For a filtered $\cO$-module $(L,F)$, let
\[ {\rm DR}^{-1}(L,F) = (L,F) \otimes_{\cO_Y} (\cD_Y,F),\]
which is an induced filtered $\cD_Y$-module.

Given a filtered differential morphism
\[ (L_1,F) \xrightarrow[]{\psi} (L_2,F),\]
we define
\[ {\rm DR}^{-1}(\psi) \colon {\rm DR}^{-1}(L_1,F) \to {\rm DR}^{-1}(L_2,F)\]
as the unique filtered $\cD_Y$-linear morphism mapping to $\psi$ under ${\rm DR}(-)$.

The functors $\widetilde{{\rm DR}}$ and ${\rm DR}^{-1}$ extend termwise to filtered differential complexes. The functors are mutually quasi-inverse provided that the filtrations are uniformly bounded-below: we have canonical filtered quasi-isomorphisms
\begin{equation}\label{eq-filtqisDRInvDR} \text{\cite{SaitoMHP}*{(2.2.10.6)}} \quad  {\rm DR}^{-1} \widetilde{\rm DR}(\cM^\bullet,F) \simeq (\cM^\bullet,F),
\end{equation}
and
\begin{equation}\label{eq-filtqisDRDRInv}  \text{\cite{SaitoMHP}*{(2.2.10.7)}} \quad \widetilde{\rm DR}{\rm DR}^{-1} (\cC^\bullet,F) \simeq (\cC^\bullet,F).
\end{equation}
\end{defi}

\subsection{Filtered Duality}
When defining the dual functor for filtered differential complexes, Saito assumes the following finiteness properties:
\[ L \text{ is coherent over } \cO_Y, \quad F_p L = 0 \text{ for all } p \ll 0, \quad F_p L = L \text{ for all } p \gg 0.\]

These assumptions are important to guarantee that the duality functor is truly a ``dual'', in the sense that it squares to give the identity functor and hence gives an equivalence of categories.

We will work in a setting that includes both the algebraic and derived de Rham complex models considered above and the filtered Spencer complexes $\widetilde{\mathrm{DR}}(\mathcal M^\bullet,F)$ appearing in Saito's theory. The construction is functorial. We will only invoke involutivity and compatibility with filtered $\mathcal D_Y$-module duality for objects coming from Saito’s theory.

To define Saito's duality for filtered differential complexes, we fix a resolution $\omega_Y[\dim Y] \to K_Y^\bullet$ by right $\cD_Y$-modules such that $K_Y^j$ is injective over $\cO_Y$ for all $j$. We also assume $K_Y^{j} = 0$ for $j < - \dim Y$ and we give $K_Y^j$ the trivial filtration: ${\rm Gr}^F_\ell K_Y^j = 0$ for $\ell \neq 0$. Moreover, as the injective dimension of $\omega_Y$ over $\cO_Y$ is at most $\dim Y$ \cite{Golovin}, we can, and will, assume that $K_Y^\bullet$ is a bounded complex. This boundedness guarantees convergence of the spectral sequences used below.

Let $(\cC^\bullet,F)$ be a filtered differential complex. For any $j, \ell, p \in \Z$, define
\[ 
F_p \cH om_{\cO_Y}^F(\cC^\ell, K_Y^j) = \left\{ \phi \in \cH om_{\cO_Y}(\cC^\ell, K_Y^j) \mid \phi(F_i \cC^\ell) \subseteq F_{p+i} K_Y^j \text{ for all }i\in \Z \right\}. 
\]

As $K_Y^j$ has the trivial filtration, this can be equivalently defined as $\cH om_{\cO_Y}(\cC^\ell/F_{-p-1},K_Y^j)$. Then we define
\[ \cH om_{\cO_Y}^F(\cC^\ell, K_Y^j) = \bigcup_{p \in \Z} F_p \cH om_{\cO_Y}^F(\cC^\ell, K_Y^j) = \varinjlim_{p} \cH om_{\cO_Y}(\cC^\ell/F_{-p-1},K_Y^j),\]
with filtration defined by $F_p$ on the right-hand side. 

The following condition ensures that the filtration on the dual complex is locally uniformly bounded below. Consider a filtered differential complex $(\cC^\bullet,F)$ such that, locally on $Y$, there exists an integer $b$, independent of the cohomological degree, satisfying
\[ F_b\cC^i=\cC^i \qquad\text{for every }i.\]

The chosen models of the algebraic and derived de Rham complexes satisfy the above hypothesis with $b=0$.

\begin{defi}[Filtered Dual of Filtered Differential Complexes]\label{def-dualityfdc} Given a filtered differential complex $(\cC^\bullet,F)$, the dual $\mathbb D(\cC^\bullet,F)$ is the filtered differential complex given as the total complex of the double complex with $(i,j)$-th  term
\[ \cH om_{\cO_Y}^F(\cC^{-i},K_Y^j).\]

The $j$-direction differential is induced by $K_Y^j \to K_Y^{j+1}$, which is filtered $\cO$-linear, hence a filtered differential morphism.

For the $i$-direction differential, we use the fact that $K_Y^j$ is a filtered $\cD_Y$-module. Indeed, by definition, the differential $(\cC^{-(i+1)},F) \xrightarrow[]{d} (\cC^{-i},F)$ is a filtered differential morphism, hence corresponds to a filtered $\cD_Y$-linear morphism
\[ {\rm DR}^{-1}(d) \colon (\cC^{-(i+1)},F) \otimes_{\cO_Y} (\cD_Y,F) \to (\cC^{-i} ,F) \otimes_{\cO_Y} (\cD_Y,F).\]

On the other hand, we have a natural isomorphism
\[ \cH om_{\cO_Y}^F((\cC^{-i},F), (K_Y^j,F)) \cong \cH om_{\cD_Y}^F((\cC^{-i},F)\otimes_{\cO_Y} ( \cD_Y,F), (K_Y^j,F)),\]
and the $i$-direction differential is defined to be the dotted arrow in the commutative diagram
\[ \begin{tikzcd} \cH om_{\cO_Y}^F(\cC^{-i},K_Y^j) \ar[r,dashed] \ar[d,"\cong"] & \cH om_{\cO_Y}^F(\cC^{-(i+1)},K_Y^j) \ar[d,"\cong"] \\  \cH om_{\cD_Y}^F((\cC^{-i},F)\otimes_{\cO_Y} (\cD_Y,F), (K_Y^j,F)) \ar[r,"{\rm DR}^{-1}(d)^*"] & \cH om_{\cD_Y}^F((\cC^{-(i+1)},F) \otimes_{\cO_Y} (\cD_Y,F), (K_Y^j,F)) \end{tikzcd}.\]

See \cite{SaitoDifferentialComplexes}*{Rmk. 2.11} for details.
\end{defi}

Our main tool to compute the filtered pieces of the dual is the double complex spectral sequence:

\begin{rmk} \label{rmk-DoubleComplexConverge} For any fixed $p$, the cohomology of the filtered piece $F_p \mathbb D(\cC^\bullet,F)$ is determined by the double complex spectral sequence
\[ E_2^{a,b} = \cH^a_K \cH^b_{\cC} F_p \cH om_{\cO_Y}^F(\cC^{-\bullet}, K_Y^\bullet) \implies \cH^{a+b} F_p \mathbb D(\cC^\bullet,F),\]
which we can rewrite as
\[ E_2^{a,b} = \cH^a_K \cH^b_{\cC} \cH om_{\cO_Y}(\cC^{-\bullet}/F_{-p-1}, K_Y^{\bullet}) \implies \cH^{a+b} F_p \mathbb D(\cC^\bullet,F).\]

Here, and throughout the paper, $\cH^\bullet_{\cC}$ is the cohomology in the $\cC^\bullet$-direction and $\cH^\bullet_K$ is the cohomology in the $K_Y^\bullet$-direction.

As $K_Y^\bullet$ is a bounded complex, this spectral sequence converges.
\end{rmk}

We conclude this subsection with the following lemma showing that the duality functor for filtered differential complexes interacts well with Grothendieck duality.

\begin{lem} \label{lem-GrDual} Let $(\cC^\bullet,F)$ be a filtered differential complex. Then, for any $p\in \Z$, there is a natural quasi-isomorphism
\[ {\rm Gr}^F_p \mathbb D(\cC^\bullet) \simeq R \cH om_{\cO_Y}({\rm Gr}^F_{-p} \cC^\bullet,\omega_Y[\dim Y]).\]
\end{lem}
\begin{proof} By definition, we form the double complex with $(i,j)$-th term
\[ \cH om_{\cO_Y}^F(\cC^{-i},K_Y^j)\]
which satisfies the equality
\[ F_p \cH om_{\cO_Y}^F(\cC^{-i},K_Y^j)  =\cH om_{\cO_Y}(\cC^{-i}/F_{-p-1},K_Y^j).\]

Thus, the double complex defining ${\rm Gr}^F_p\mathbb D(\cC^\bullet)$ has $(i,j)$-term $\cH om_{\cO_Y}({\rm Gr}^F_{-p}\cC^{-i},K_Y^j)$. A computation gives that the $i$-direction differential is the $\cO$-linear dual of the $\cO$-linear morphism ${\rm Gr}^F_{-p} \cC^{-(i+1)} \to {\rm Gr}^F_{-p} \cC^{-i}$. The claim follows by the definition of the Grothendieck duality functor.
\end{proof}

We record the following consequence, which holds even though $\mathbb D$ does not satisfy $\mathbb D^2 \simeq {\rm Id}$.

\begin{cor} \label{cor-DoubleDual} Let $\varphi \colon (\cC^\bullet_1,F) \to (\cC^\bullet_2,F)$ be a morphism between filtered differential complexes. If
\[\cC^\bullet_1/ F_{-p-1} \cC^\bullet_1 \to \cC^\bullet_2/F_{-p-1}\cC^\bullet_2 \text{ is a quasi-isomorphism for all } p \leq m,\]
then
\[F_p \mathbb D(\cC^\bullet_2)\to F_p\mathbb D(\cC^\bullet_1) \text{ is a quasi-isomorphism for all } p \leq m.\]

The converse holds if we assume
\begin{equation} \label{eq-AssumeGrCoh1} {\rm Gr}^F_p \cC^\bullet_i \in D^b_{\rm coh}(\cO_Y) \text{ for } p\in \Z, i\in \{1,2\},\end{equation}
and that there exists $\ell  \in \Z$ such that
\begin{equation} \label{eq-AssumeGrCoh2} F_\ell \cC^\bullet_i = \cC^\bullet_i \text{ for } i \in \{1,2\}.\end{equation}

Under the assumptions \eqref{eq-AssumeGrCoh1} and \eqref{eq-AssumeGrCoh2}, either condition is equivalent to
\[ {\rm Gr}^F_{-p}(\varphi) \text{ is a quasi-isomorphism for all } p \leq m.\]
\end{cor}
\begin{proof} The first claim follows by construction of the filtered duality functor.

The converse (and equivalent reformulation in terms of ${\rm Gr}^F_{-p}(\varphi)$) follow by induction on $p$, using bi-duality for the usual Grothendieck duality functor and the computation in \lemmaref{lem-GrDual}.
\end{proof}

\subsection{Duality and the Ext Filtration}\label{ss-DualityExt}
The key duality computation is the following.

\begin{prop} \label{prop-DualOfDR} Let $(\cC^\bullet,F) = \widetilde{\rm DR}(\cM,F)[-\dim Y]$ for a filtered left $\cD_Y$-module $(\cM,F)$ as in Remark \ref{rmk-LeftDMods}. Assume $F_p \cM = \cM$ for $p \gg 0$.

Then there is a natural isomorphism
\[ \cH^qF_p {\rm DR}^{-1} \mathbb D(\cC^\bullet,F) \cong \cE xt^{q+\dim Y}_{\cO_Y}(\cM/F_{-p-1},\omega_Y).\]
\end{prop}
\begin{proof} Note that we have
\[ (\cC^i,F) = \Omega_Y^i \otimes_{\cO_Y} (\cM,F[-i]).\]

Hence, we can write
\[ F_p \cH om_{\cO_Y}^F(\cC^{-i},K_Y^j) = \cH om_{\cO_Y}(\cC^{-i}/F_{-p-1},K_Y^j)\]
and $\cC^{-i}/F_{-p-1} = \Omega_Y^{-i} \otimes_{\cO_Y} (\cM/F_{-p-i-1})$, so we have
\[ F_p \cH om_{\cO_Y}^F(\cC^{-i},K_Y^j) = \cH om_{\cO_Y}(\cM/F_{-p-i-1},K_Y^j) \otimes_{\cO_Y} \bigwedge^{-i} \cT_Y.\]

Recall that $\mathbb D(\cC^\bullet,F)$ is defined as the total complex of the double complex whose $(i,j)$-th term is $\cH om_{\cO_Y}^F(\cC^{-i},K_Y^j)$, and so similarly, $F_p {\rm DR}^{-1}\mathbb D(\cC^\bullet,F)$ is the total complex of the double complex whose $(i,j)$-th term is $F_p {\rm DR}^{-1}\cH om_{\cO_Y}^F(\cC^{-i},K_Y^j)$. We will use the spectral sequence
\[ E_2^{a,b} = \cH^a_{K} \cH^b_{\cC}F_p {\rm DR}^{-1} \cH om_{\cO_Y}^F(\cC^{-\bullet},K_Y^*) \implies \cH^{a+b} F_p {\rm DR}^{-1}\mathbb D(\cC^\bullet,F).\]

Define for each $j\in \Z$ a filtered $\cO$-module $(\cN^j,F)$ as $F_p \cN^j = \cH om_{\cO_Y}(\cM/F_{-p-1},K_Y^j)$, so that $\cN^j = \bigcup_{p\in \Z} F_p \cN^j \subseteq \cH om_{\cO_Y}(\cM, K_Y^j)$. Following the above notation, we could write $\cN^j = \cH om_{\cO_Y}^F(\cM,K_Y^j)$.

Recall that $\cH om_{\cO_Y}(\cM,K_Y^j)$ is naturally a right $\cD_Y$-module by the action $(\psi \theta)(m) = \psi( \theta(m)) + \psi(m)\theta$ for any $\theta \in \cT_Y$. It is easy to check that $\cN^j$ is a right $\cD_Y$-submodule, and that the filtration is compatible with the action of differential operators. Indeed, if $\psi(F_{-p-1}\cM) = 0$, we want to check that $(\psi \theta)(F_{-p-2} \cM) = 0$, too. But given $m\in F_{-p-2} \cM$, we have 
\[
(\psi \theta)(m) = \psi(\theta(m)) + \psi(m)\theta \in \psi(F_{-p-1}\cM) + \psi(F_{-p-2}\cM)\theta = 0.
\]

Importantly, the assumption $F_p \cM = \cM$ for $p \gg 0$ implies that $F_p \cN^j = 0$ for $p \ll 0$.

We can rephrase the $\cC^\bullet$-direction differential as a morphism
\[ (\cN^j,F[-i]) \otimes_{\cO_Y} \bigwedge^{-i} \cT_Y \to ( \cN^j,F[-i-1]) \otimes_{\cO_Y} \bigwedge^{-(i+1)} \cT_Y,\]
and in fact, a computation in local coordinates shows that this morphism agrees with the differential in $\widetilde{\rm DR}(\cN^j,F)$. Thus, when we apply ${\rm DR}^{-1}(-)$, the resulting complex is ${\rm DR}^{-1} \widetilde{\rm DR}(\cN^j,F)$. As explained in \lemmaref{lem-InducedResolution}, this complex is filtered quasi-isomorphic to $(\cN^j,F)$.

Hence, 
\[\cH^b_{\cC} F_p {\rm DR}^{-1} \cH om_{\cO_Y}^F(\cC^{-\bullet},K_Y^*) \cong \begin{cases} 0 & b \neq 0 \\ F_p \cN^* & b =0\end{cases}.\] Using the fact that $F_p \cN^j = \cH om_{\cO_Y}(\cM/F_{-p-1}, K_Y^j)$, with differential induced by that of $K_Y^\bullet$, we see that the complex computing $\cH^a_K \cH^b_{\cC}F_p {\rm DR}^{-1} \cH om_{\cO_Y}^F(\cC^{-\bullet},K_Y^*)$ is\small
\[ \dots \to \cH om_{\cO_Y}(\cM/F_{-p-1},K_Y^{j-1}) \to \cH om_{\cO_Y}(\cM/F_{-p-1},K_Y^{j}) \to \cH om_{\cO_Y}(\cM/F_{-p-1},K_Y^{j+1}) \to \cdots,\]\normalsize
which completes the proof.
\end{proof}

\begin{cor} \label{cor-ConstructExtFilt} The morphism
\[ \cH^q F_p {\rm DR}^{-1} \mathbb D(\Omega_{X/\C}^H, {\rm Fil}^{\rm inf}_\bullet) \to \cH^q {\rm DR}^{-1} \mathbb D(\Omega_{X/\C}^H, {\rm Fil}^{\rm inf}_\bullet)\]
is naturally identified with the morphism
\[ \cE xt^{q+\dim Y}_{\cO_Y}(\cO_Y/I^{p+1},\omega_Y) \to \cH^{q+\dim Y}_X(\omega_Y)\]
for any $p,q \in \Z$. The analogous statement holds with the symbolic infinitesimal filtration.
\end{cor}
\begin{proof} This follows immediately from the proof of \propositionref{prop-DualOfDR} above, noting that in this case, the module $(\cN^j,F)$ can be identified with $(\Gamma_X(K_Y^j),O)$, where $O_\bullet \Gamma_X(K_Y^j)$ is the order filtration.
\end{proof}

\begin{proof}[Proof of \theoremref{thm-InterpretExt}] Apply the side-changing functor $- \otimes_{\cO_Y} \omega_Y^{-1}$, with the appropriate shift of notation, to the (shifted) complex
\[ {\rm DR}^{-1} \mathbb D(\Omega_{X/\C}^H,{\rm Fil}^{\rm inf}_\bullet)[-\dim Y]\]
of \corollaryref{cor-ConstructExtFilt}. The resulting filtered $\cD_Y$-module complex is $(\cE^\bullet(X),F)$.
\end{proof}

\subsection{Hodge-theoretic input: Strictness and the Du Bois complex}\label{ss-HodgeInput}

Saito \cite{SaitoMHP}*{Sec. 2.4} defines the dual for filtered $\cD$-modules in a way compatible with the above only if $(\cM,F) = (L,F) \otimes_{\cO_Y} (\cD_Y,F)$ for $F_p L = L$ for $p \gg 0$ and $F_p L = 0$ for $p \ll 0$.

For a \emph{coherent} filtered $\cD_Y$-module $(\cM,F)$, \cite{SaitoMHP}*{Prop. 2.1.16} shows that one can always take a resolution by such induced filtered $\cD$-modules, and then one can use the functors defined above to compute the filtered dual of $(\cM,F)$. In fact, this resolution can be chosen locally free of finite rank over $\cD_Y$.

\begin{prop}[\cite{SaitoMHP}*{Lem. 5.1.13} and \cite{SaitoMHM}*{Prop. 2.6}] \label{prop-strictnessDual} Let $(\cM^\bullet,F)$ be a complex of filtered $\cD_Y$-modules underlying a bounded complex of mixed Hodge modules. Then the dual $\mathbf D(\cM^\bullet,F)$ is a bounded-below complex of filtered $\cD$-modules which is strict.
\end{prop}

In our setting, strictness means the following: if $(\cN^\bullet,F)$ is a resolution by locally free induced filtered $\cD_Y$-modules of finite rank   as in the previous paragraph, then the natural morphism
\[ \cH^q F_p \mathbf D(\cN^\bullet,F) \to \cH^q \mathbf D(\cN^\bullet) = \cH^q \mathbf D(\cM^\bullet)\]
is injective, and its image is equal to $F_p \cH^q \mathbf D(\cN^\bullet)$.

Now, if $X\subseteq Y$ is embedded into a smooth variety, we can represent the constant mixed Hodge module object $\Q_X^H \in D^b({\rm MHM}(X))$ as some bounded complex of mixed Hodge modules on $Y$: $\iota_* \Q_X^H = M^\bullet$. Saito shows that this is compatible with the filtered Du Bois complex which allows him to prove compatibility of his theory with Deligne's mixed Hodge theory \cite{SaitoMHC}*{Cor. 4.3}.

\begin{thm}[\cite{SaitoMHC}*{Thm. 4.2}; see also \cite{DBSaito}*{Thm. 1.5}] \label{thm-DuBoisCompare} There is a canonical filtered quasi-isomorphism
\[ \widetilde{\rm DR}(\cM^\bullet,F) \simeq (\underline{\Omega}_X^\bullet, F)\]
in the bounded derived category of filtered differential complexes on $Y$ (defined above). Here $(\cM^\bullet,F)$ is the complex of filtered $\cD_Y$-modules underlying $M^\bullet = \iota_* \Q_X^H$.

In particular, we have a quasi-isomorphism for all $p\in \Z$:
\begin{equation} \label{eq-DuBoisGrF} {\rm Gr}^F_{-p} {\rm DR}(\Q_X^H) \simeq \underline{\Omega}_X^p[-p].\end{equation}
\end{thm}

\section{Proof of Main Results}
We use the notation of Subsection \ref{subsec-Setup}.

\subsection{Hodge and Ext Filtrations}
We begin with two useful lemmas concerning local cohomology.

The following formal observation will be useful to construct morphisms to local cohomology below:
\begin{lem} \label{lem-adjunction} Let $\cM^\bullet \in D(\cD_Y)$ satisfy $j^*(\cM^\bullet) =0$, where $j\colon U = Y\setminus X \hookrightarrow Y$ is the open embedding. 

Then any morphism $\cM^\bullet \to \omega_Y$ in $D(\cD_Y)$ factors uniquely as
\[ \cM^\bullet \to R\Gamma_X(\omega_Y) \to \omega_Y.\]
\end{lem}
\begin{proof} This is a standard argument. We have, in $D^b(\cD_Y)$, the exact triangle
\[R\Gamma_X(\omega_Y) \to \omega_Y \to Rj_* j^*(\omega_Y) \xrightarrow[]{+1},\]
see \cite{HTT}*{Prop. 1.7.1}. The desired (unique, in this case) factorization follows from standard properties of triangulated categories.
\end{proof}

We have the following well-known consequence of the Riemann-Hilbert correspondence.

\begin{lem} \label{lem-endConstant} There is a canonical isomorphism
\[ {\rm End}_{D(\cD_Y)}(R\Gamma_X(\omega_Y)) \cong H^0(X^{\rm an},\C).\]
In particular, if $\chi \in {\rm End}_{D(\cD_Y)}(R\Gamma_X(\omega_Y))$ is nonzero on each connected component, then $\chi$ is an automorphism.
\end{lem}
\begin{proof} The isomorphism is most easily seen after applying the Riemann-Hilbert functor, so that
\[ {\rm End}_{D(\cD_Y)}(R\Gamma_X(\omega_Y)) \cong {\rm End}_{D(\C_{Y^{\rm an}})}(\mathbf D^{\rm Verd}(\C_{X^{\rm an}})) \cong {\rm End}_{D(\C_{Y^{\rm an}})}(\C_{X^{\rm an}}),\]
which is well known to be isomorphic to $H^0(X^{\rm an},\C)$. Here $\mathbf D^{\rm Verd}(-)$ is Verdier duality for constructible complexes.

For the last claim, note that $H^0(X^{\rm an},\C) \cong \C^{\rho}$, where $\rho$ is the number of connected components of $X^{\rm an}$. Hence, an endomorphism is invertible if and only if it is invertible on each connected component, on which it is constant (so that nonzero is equivalent to invertible).
\end{proof}

The Hodge filtration on local cohomology can be studied through the filtration on the Du Bois complex. We have the identification
\[ \iota_* \iota^! \Q_Y^H[\dim Y] = \iota_* \mathbf D_X(\Q_X^H[\dim Y])(-\dim Y)\]
and we note that the right $\cD_Y$-module underlying $\cH^q (\iota_* \iota^! \Q_Y^H[\dim Y])$ agrees with the local cohomology module $\cH^q_X(\omega_Y)$. Thus, by applying ${\rm Gr}^F_{p-\dim Y}{\rm DR}(-)$ and using the isomorphism \eqref{eq-DuBoisGrF}, we have a relation between local cohomology and the (dual) Du Bois complex of $X$:
\begin{equation} \label{eq-LocCohDuBois} {\rm Gr}^F_{p-\dim Y} {\rm DR}(\iota_* \iota^! \Q_Y^H[\dim Y]) \simeq \iota_* \mathbb D_X^{\rm coh}(\underline{\Omega}_X^p)[p-\dim Y].\end{equation}

We now use this description and filtered duality to prove that the Hodge filtration is contained in the symbolic (hence, usual) Ext filtration. Our argument also gives a second proof that the (symbolic) Ext filtration on local cohomology is compatible with the order filtration on $\cD_Y$.

\begin{proof}[Proof of \theoremref{thm-HodgeInExt}] We start with the filtered morphisms from Remark \ref{rmk-MapDuBois}:
\[ (\Omega_{X/\C}^H, {\rm Fil}_{\rm inf}^\bullet) \to(\Omega_{X/\C}^H, {\rm Fil}_{\rm symbinf}^\bullet) \to(\underline{\Omega}_{X}^\bullet, F^\bullet), \]
which are morphisms of filtered differential complexes. Indeed, we can see this by factoring through $(\Omega_X^\bullet,F^\bullet)$, and the morphisms in and out of that complex are filtered $\cO_Y$-linear, so they are automatically morphisms of filtered differential complexes.

We may and will assume that $Y$ is affine. Using \theoremref{thm-DuBoisCompare} we have a filtered quasi-isomorphism $(\underline{\Omega}_X^\bullet,F^\bullet) \simeq \widetilde{\rm DR}(\cM^\bullet,F)$ where $(\cM^\bullet,F)$ is a bounded complex of induced filtered $\cD_Y$-modules underlying $\iota_*\Q_X^H \in D^b({\rm MHM}(Y))$. By coherence of the filtrations, we can assume that it is a bounded complex of finite locally free induced filtered $\cD_Y$-modules, to apply \propositionref{prop-strictnessDual}.

Apply the contravariant functor $\cH^q F_p {\rm DR}^{-1} \mathbb D(-)$ to these morphisms. We get morphisms
\[\cH^q F_p {\rm DR}^{-1} \mathbb D(\widetilde{\rm DR}(\cM^\bullet,F)) \to \cH^q F_p {\rm DR}^{-1} \mathbb D(\Omega_{X/\C}^H,{\rm Fil}_{\rm symbinf}^\bullet) \to \cH^q F_p {\rm DR}^{-1} \mathbb D(\Omega_{X/\C}^H,{\rm Fil}_{\rm inf}^\bullet).\]

We use the commutativity \cite{SaitoMHP}*{2.4.11, p. 895} ${\rm DR}^{-1} \circ \mathbb D(-) \simeq \mathbf D\circ  {\rm DR}^{-1}(-)$ for the first object, to rewrite it as
\[ 
    \cH^q F_p \mathbf D {\rm DR}^{-1} \widetilde{\rm DR}(\cM^\bullet,F) \cong \cH^q F_p \mathbf D(\cM^\bullet,F)
\]
where we used the fact that ${\rm DR}^{-1}$ and ${\rm DR}$ are mutually quasi-inverse on complexes of induced filtered $\cD_Y$-modules. By the definition of the dual functor, $\mathbf D(\cM^\bullet,F)$ underlies the dual object $\iota_* \mathbf D(\Q_X^H)$, which we noted at the start of this subsection has an isomorphism
\[ \iota_* \mathbf D(\Q_X^H) \simeq (\iota_* \iota^! \Q_Y^H[\dim Y])(\dim Y)[\dim Y].\]

Hence, we have a filtered isomorphism
\[ 
    \cH^q \mathbf D(\cM^\bullet,F) \cong (\cH^{q+\dim Y}_X(\omega_Y),F_{\bullet-\dim Y}).
\]

On the other hand, by construction, there are canonical equalities
\[ (\Omega_{X/\C}^H, {\rm Fil}_{\rm inf}^\bullet) = {\rm DR}(\widehat{\cO}_Y, I^\bullet)[-\dim Y],\]
\[ (\Omega_{X/\C}^H, {\rm Fil}_{\rm symbinf}^\bullet) = {\rm DR}(\widehat{\cO}_Y, I^{(\bullet)})[-\dim Y],\]
where the (symbolic) $I$-adic filtration on the completion is equal to the entire completion for $p =0$. Hence, we can apply the result of \propositionref{prop-DualOfDR} and rewrite the morphism
\[
    \cH^q F_p {\rm DR}^{-1} \mathbb D(\Omega_{X/\C}^H,{\rm Fil}_{\rm symbinf}^\bullet) \to \cH^q F_p {\rm DR}^{-1} \mathbb D(\Omega_{X/\C}^H,{\rm Fil}_{\rm inf}^\bullet) 
\]
as the natural morphism
\[ \cE xt^{q+\dim Y}(\cO_Y/I^{(p+1)}, \omega_Y) \to \cE xt^{q+\dim Y}(\cO_Y/I^{p+1},\omega_Y).\]

Replacing $q$ by $q-\dim Y$ and using strictness of the dual, combined with the computation above, we have a natural composition of morphisms:
\[ F_{p-\dim Y} \cH^q_X(\omega_Y) \cong \cH^{q-\dim Y} F_p \mathbf D(\cM^\bullet,F) \to \cE xt^q_{\cO_Y}(\cO_Y/I^{(p+1)},\omega_Y) \to \cE xt^q_{\cO_Y}(\cO_Y/I^{p+1},\omega_Y).\]

These morphisms fit into the commutative diagram:
\[ \begin{tikzcd} F_{p-\dim Y} \cH^q_X(\omega_Y) \ar[r] \ar[d] & \cE xt^q_{\cO_Y}(\cO_Y/I^{(p+1)},\omega_Y) \ar[r] \ar[d] & \cE xt^q_{\cO_Y}(\cO_Y/I^{p+1},\omega_Y)\ar[dl] \\ \cH^{q-\dim Y} \mathbf D(\cM^\bullet) \ar[r,"\chi"] & \cH^{q-\dim Y} {\rm DR}^{-1}\mathbb D(\Omega_{X/\C}^H) & {} \end{tikzcd}.\]

The proof is complete once we observe that the morphism $\chi$ is an isomorphism. Indeed, working in the derived category of $\cD$-modules, the morphism
\begin{equation}\label{eq:drd}
{\rm DR}^{-1} \mathbb D(\underline{\Omega}_X^\bullet, F^\bullet) \to {\rm DR}^{-1} \mathbb D(\Omega_{X/\C}^H,{\rm Fil}_{\rm inf}^\bullet
)
\end{equation}
can be identified with an endomorphism of $R\Gamma_X(\omega_Y)[\dim Y]$.

It suffices by \lemmaref{lem-endConstant} to show that the morphism~\eqref{eq:drd} is nonzero on any connected component of $X$. It follows from \cite{BhattDDR}*{Rmk. 5.3} that the morphism~\eqref{eq:drd} is an identity on $Y\setminus X_{\rm sing}$, hence we are done.
\end{proof}

\subsection{Higher Injectivity Theorems}

We turn now to the proof of the higher injectivity theorem (\theoremref{thm-Injectivity}). We make use of some elementary results:
\begin{lem} \label{lem-testIsoMod} Let $C\subseteq A \subseteq B$ be a chain of monomorphisms in an abelian category. Then the inclusion $A \to B$ is an epimorphism if and only if it is an isomorphism if and only if the induced morphism $A/C \to B/C$ is an isomorphism.
\end{lem}

The following will be used in the spectral sequence arguments below:

\begin{lem} \label{lem-ControlCohomology} Consider a commutative diagram
\[ \begin{tikzcd} A^{-1}\ar[r,"d"] \ar[d,"\alpha"] & A^0 \ar[r,"d"] \ar[d,"\chi"] & A^1 \ar[d,"\beta"] \\ B^{-1} \ar[r,"d"] & B^0 \ar[r,"d"] & B^1 \end{tikzcd}\]
where both compositions of horizontal arrows vanish.

\begin{enumerate}\item  If $\alpha, \chi$ are isomorphisms and $\beta$ is injective, then the induced map $\cH^0(A^\bullet) \to \cH^0(B^\bullet)$ is an isomorphism.

\item If $A^1 =B^1 = 0$, $\alpha$ is an isomorphism, and $\chi$ is injective, then the induced map $\cH^0(A^\bullet) \to \cH^0(B^\bullet)$ is injective.
\end{enumerate}
\end{lem}
\begin{proof} Assume $\alpha$ and $\chi$ are both isomorphisms and $\beta$ is injective. By \lemmaref{lem-testIsoMod} above, $\cH^0(\chi)$ is an isomorphism if and only if the morphism $\ker(d_A) \to \ker(d_B)$ is an isomorphism. As $\chi$ is injective, it suffices to prove that this morphism is surjective. Let $b \in \ker(d_B)$, so $d_B(b) = 0$. As $\chi$ is surjective, there exists $a\in A^0$ such that $\chi(a) = b$. We need to prove $d_A(a) = 0$. We have the chain of equalities $\beta d_A(a) = d_B \chi(a) = d_B(b) = 0$, and so this follows from injectivity of $\beta$.

For the second claim, it suffices to show that if $a\in \ker(d_A) = A^0$ satisfies $\chi(a) \in {\rm im}(d_B)$, then $a\in {\rm im}(d_A)$. Write $\chi(a) = d_B(b)$, and so using the fact that $\alpha$ is an isomorphism, we can find $a' \in A^{-1}$ such that $b = \alpha(a')$. Then $\chi(a) = d_B(\alpha(a')) = \chi(d_A(a'))$, and so injectivity of $\chi$ gives the claim.
\end{proof}

We will also use the following slightly more technical lemma. If, in the following lemma, we had a morphism of filtered $\cD_Y$-modules in place of a morphism of complexes, the proof would be essentially immediate (and the result is well known to experts).

\begin{lem} \label{lem-Technical} Let $\phi \colon (\cM_2^\bullet,F) \to (\cM_1^\bullet,F)$ be a morphism between filtered complexes of $\cD_Y$-modules such that, locally on $Y$, there exists an integer $a$ such that \[ F_p\mathcal M_\ell^i=0 \quad\text{for every }p<a,\quad i\in\mathbb Z,\quad \ell\in\{1,2\}. \]

Fix $k\in\Z$, and assume that the following hold:
\begin{enumerate} \item The morphism $F_p \widetilde{\rm DR}(\cM_2^\bullet) \to F_p \widetilde{\rm DR}(\cM_1^\bullet)$ is a quasi-isomorphism for all $p < k$.

\item $ \cH^qF_p \cM_2^\bullet \to \cH^q F_p \cM_1^\bullet$ is injective for all $q\in \Z$ and $p \leq k$.
\end{enumerate}

Then the following statements hold:
\begin{enumerate} \item For all $q \in \Z$, the morphism
\[ \cH^q {\rm Gr}^F_k \widetilde{\rm DR}(\cM_2^\bullet) \to \cH^q {\rm Gr}^F_k \widetilde{\rm DR}(\cM_1^\bullet)\]
is injective.

\item For all $p < k$ and $q\in \Z$, the morphism $ \cH^qF_p \cM_2^\bullet \to \cH^qF_p \cM_1^\bullet$ is an isomorphism.
\end{enumerate}
    
\end{lem}
\begin{proof} We first prove that for all $p < k$ and $q\in \Z$, the morphism $\cH^qF_p \cM_2^\bullet \to \cH^qF_p \cM_1^\bullet$ is an isomorphism. We use induction on $p$. The claim is clearly true for $p \ll 0$ as both modules are $0$ in that case by assumption. So fix $p < k$ and assume for all $p' < p$ that $\cH^qF_{p'} \cM_2^\bullet \to \cH^qF_{p'} \cM_1^\bullet$ is an isomorphism. 

\textbf{Claim 0.} It suffices to prove the injective morphism $\cH^q {\rm Gr}^F_p(\phi)\colon  \cH^q {\rm Gr}^F_p\cM_2^\bullet \to \cH^q {\rm Gr}^F_p \cM_1^\bullet$ is surjective.

\textbf{Proof.} By one of the Four Lemmas, we see that the morphism $\cH^q {\rm Gr}^F_p(\phi)$ is always injective as in the Claim statement. Now, if we assume it is surjective, we use the other Four Lemma to conclude that $\cH^q F_p \cM_2^\bullet \to \cH^q F_p \cM_1^\bullet$ is surjective, hence an isomorphism.

The morphism $\phi$ induces a morphism of double complexes ${\rm Gr}^F_p {\rm DR}(\cM^\bullet_2) \to {\rm Gr}^F_p {\rm DR}(\cM^\bullet_1)$. Specifically, for $\ell \in \{1,2\}$, the double complex has $(s,t)$-term
\[ {\rm Gr}^F_{p+t} \cM^s_\ell \otimes_{\cO_Y} \bigwedge^{-t} \cT_Y.\]

If we first take cohomology in the $s$-direction, then since $\bigwedge^{-t} \cT_Y$ is flat over $\cO_Y$, we get
\[ (\cH^q {\rm Gr}^F_{p+t} \cM^\bullet_\ell)\otimes_{\cO_Y} \bigwedge^{-t} \cT_Y\]
and the remaining $t$-direction differential is that from the de Rham complex construction. In other words, we have the morphism of $E_1$-page spectral sequences
\[ \begin{tikzcd}  E_1^{a,q} = \cH^q {\rm Gr}^F_{p+a} \cM^\bullet_2 \otimes_{\cO_Y} \bigwedge^{-a} \cT_Y\ar[r, Rightarrow] \ar[d,swap,"\cH^q {\rm Gr}^F_{p+a}(\phi) \otimes {\rm id}"] & \cH^{a+q} {\rm Gr}^F_p {\rm DR}(\cM_2^\bullet) \ar[d,"\cH^{a+q} {\rm Gr}^F_p {\rm DR}(\phi)"] \\ \widetilde{E}_1^{a,q} = \cH^q {\rm Gr}^F_{p+a} \cM^\bullet_1 \otimes_{\cO_Y} \bigwedge^{-a} \cT_Y\ar[r, Rightarrow] &  \cH^{a+q} {\rm Gr}^F_p {\rm DR}(\cM_1^\bullet)\end{tikzcd} \]
where our assumption implies that the morphism on the abutments is an isomorphism for any $a,q$. By the inductive hypothesis, note that the morphism $E_1^{a,q} \to \widetilde{E}_1^{a,q}$ is an isomorphism for $a < 0$ and injective for $a =0$ for any $q\in \Z$. This persists through all pages of the spectral sequences, as we now show:

\textbf{Claim 1.} For any $r\geq 1$, the morphism $E_r^{a,q} \to \widetilde{E}_r^{a,q}$ is $\begin{cases} \text{an isomorphism } & a < 0 \\ \text{injective } & a = 0\end{cases}$.

\textbf{Proof.} We use induction on $r$. Assume the claim is true for the $E_r$-page. Then, since $E_{r+1}$ is the cohomology of $(E_r,d_r)$, consider the following commutative diagram:
\[ \begin{tikzcd} E_r^{a-r,q+r-1} \ar[r] \ar[d] & E_r^{a,q} \ar[r] \ar[d] & E_r^{a+r,q-r+1}\ar[d] \\ \widetilde{E}_r^{a-r,q+r-1} \ar[r] & \widetilde{E}_r^{a,q} \ar[r] & \widetilde{E}_r^{a+r,q-r+1}\end{tikzcd}.\]

Hence, the claim is immediate from \lemmaref{lem-ControlCohomology} above.

\textbf{Claim 2.} The morphism $E_\infty^{0,q} \to \widetilde{E}_\infty^{0,q}$ is surjective, hence by Claim 1 is an isomorphism.

\textbf{Proof.} By assumption, the morphism $A^q \to \widetilde{A}^q$ between the abutments of the two spectral sequences is an isomorphism for any $q\in \Z$. Moreover, if $G^\bullet A^q$ (resp. $G^\bullet \widetilde{A}^q$) is the filtration given by the spectral sequence, with associated graded pieces ${\rm Gr}_G^a A^q = E_\infty^{a,q-a}$ (resp. ${\rm Gr}_G^a \widetilde{A}^q = \widetilde{E}_\infty^{a,q-a}$), then this morphism respects the filtration.

Let $m \in \widetilde{E}_\infty^{0,q} = {\rm Gr}_G^0 \widetilde{A}^q = G^0 \widetilde{A}^q$, where for the last equality, we use that $0$ is the largest possible index $a$ with ${\rm Gr}_G^a \widetilde{A}^q \neq 0$. By surjectivity of $A^q \to \widetilde{A}^q$, we can find $\eta \in A^q$ mapping to $m$. We are done if $\eta \in G^0 A^q$. If not, $\eta \in G^a A^q$ for some $a < 0$ and we choose $a<0$ maximal such that $\eta\in G^a A^q$; then $\eta$ defines a nonzero class in ${\rm Gr}_G^aA^q$. But then under the morphism ${\rm Gr}_G^a A^q \to {\rm Gr}_G^a \widetilde{A}^q$, the class $\overline{\eta}$ maps to $0$, which contradicts the isomorphism property of this morphism. This finishes the proof.

\textbf{Claim 3.} The morphism $\cH^q F_p \cM^\bullet_2 \to \cH^q F_p\cM^\bullet_1$ is an isomorphism.

\textbf{Proof.} By \textbf{Claim 0}, it suffices to prove that $\cH^q{\rm Gr}^F_p  \cM^\bullet_2 \to  \cH^q {\rm Gr}^F_p \cM^\bullet_1$ is an isomorphism. This is equivalent to the claim that $E_1^{0,q} \to \widetilde{E}_1^{0,q}$ is an isomorphism. By applying \lemmaref{lem-testIsoMod} again, it suffices to prove $E_2^{0,q} \to \widetilde{E}_2^{0,q}$ is an isomorphism. Repeating in this way, we reduce to the claim that $E_\infty^{0,q} \to \widetilde{E}_\infty^{0,q}$ is an isomorphism, which we observed above.

\textbf{Claim 4.} For all $\ell \in \Z$, the morphism
\[ \cH^\ell {\rm Gr}^F_k \widetilde{\rm DR}(\cM_2^\bullet) \to \cH^\ell {\rm Gr}^F_k \widetilde{\rm DR}(\cM_1^\bullet)\]
is injective.

\textbf{Proof.} By the above, we know that $ \cH^q {\rm Gr}^F_p \cM_2^\bullet \to \cH^q {\rm Gr}^F_p\cM_1^\bullet$ is an isomorphism for all $p < k$. We now focus on the morphism of spectral sequences
\[ \begin{tikzcd}  E_1^{a,q} = \cH^q {\rm Gr}^F_{k+a} \cM^\bullet_2 \otimes_{\cO} \bigwedge^{-a} \cT_Y\ar[r, Rightarrow] \ar[d] & \cH^{a+q} {\rm Gr}^F_k {\rm DR}(\cM_2^\bullet) \ar[d] \\ \widetilde{E}_1^{a,q} = \cH^q {\rm Gr}^F_{k+a} \cM^\bullet_1 \otimes_{\cO} \bigwedge^{-a} \cT_Y\ar[r, Rightarrow] &  \cH^{a+q} {\rm Gr}^F_k {\rm DR}(\cM_1^\bullet)\end{tikzcd} \]
to which we can apply \textbf{Claim 1} above and conclude that for any $r\geq 1$, the morphism $E_r^{a,q} \to \widetilde{E}_r^{a,q}$ is $\begin{cases} \text{an isomorphism } & a < 0 \\ \text{injective } & a = 0\end{cases}$. 

Our goal is to prove, for all $\ell \in \Z$, that the morphism of abutments $\psi \colon A^\ell \to \widetilde{A}^\ell$ is injective. Once again, we use the filtration induced by the spectral sequences. By \textbf{Claim 1}, the morphism ${\rm Gr}_G^a(\psi) \colon {\rm Gr}_G^a A^\ell = E_\infty^{a,\ell-a} \to \widetilde{E}_\infty^{a,\ell-a} = {\rm Gr}_G^a \widetilde{A}^\ell$ is injective for all $a$, and so the claimed injectivity is obvious.
\end{proof}

This lemma is useful because it translates injectivity at the filtered $\cD$-module level into injectivity for filtered differential complexes.

We prove the following stronger version of \theoremref{thm-Injectivity}.

\begin{thm} \label{thm-Injectivity2} Let $X$ be a complex algebraic variety and let $k \in \Z_{\geq 0}$.

\begin{enumerate} \item \label{itm-InjectivePre} If $X$ is pre-$(k-1)$-Du Bois, then the natural morphism
\[ \cE xt_{\cO_X}^q(\underline{\Omega}_X^k,\omega_X^\bullet) \to \cE xt_{\cO_X}^q(\cH^0(\underline{\Omega}_X^k),\omega_X^\bullet)\]
is injective for all $q\in \Z$.
\item If $X$ is strictly $(k-1)$-Du Bois, then the natural morphism
\[ \cE xt_{\cO_X}^q(\underline{\Omega}_X^k,\omega_X^\bullet) \to \cE xt_{\cO_X}^q(\Omega_X^k,\omega_X^\bullet)\]
is injective for all $q\in \Z$.

\item\label{thm-symbolic} If ${\rm Gr}_{\rm symbinf}^p \Omega_{X/\C}^H \to \underline{\Omega}_X^p[-p]$ is a quasi-isomorphism for all $p\leq k-1$, then the natural morphism
\[ \cE xt_{\cO_X}^q(\underline{\Omega}_X^k,\omega_X^\bullet) \to \cE xt_{\cO_X}^q({\rm Gr}_{\rm symbinf}^k \Omega_{X/\C}^H[k],\omega_X^\bullet)\]
is injective for all $q\in \Z$.


\item\label{thm-strong} If ${\rm Gr}_{\rm inf}^p \Omega_{X/\C}^H \to \underline{\Omega}_X^p[-p]$ is a quasi-isomorphism for all $p\leq k-1$, then the natural morphism
\[ \cE xt_{\cO_X}^q(\underline{\Omega}_X^k,\omega_X^\bullet) \to \cE xt_{\cO_X}^q({\rm Gr}_{\rm inf}^k \Omega_{X/\C}^H[k],\omega_X^\bullet)\]
is injective for all $q\in \Z$.


\item  \label{itm-Cotangent} If $\bigwedge^p L_{X/\C} \to \underline{\Omega}_X^p$ is a quasi-isomorphism for all $p \leq k-1$, then the natural morphism
\[ \cE xt^q_{\cO_X}(\underline{\Omega}_X^k,\omega_X^\bullet) \to \cE xt^q_{\cO_X}(\bigwedge^k L_{X/\C}, \omega_X^\bullet)\]
is injective for all $q\in \Z$.
\end{enumerate}
\end{thm}

\begin{rmk}
    The morphisms in \theoremref{thm-Injectivity2}(\ref{thm-symbolic}) and \theoremref{thm-Injectivity2}(\ref{thm-strong}) are studied in an upcoming article where the authors propose new higher singularity classes related to these constructions \cite{CDOHigher}.
\end{rmk}

\begin{rmk} \label{rmk-StupidlyFiltered} We have noted that $(\underline{\Omega}_X^\bullet,F)$, $(\Omega_{X/\C}^H, {\rm Fil}_{\rm inf}^\bullet)$, and $(\Omega_{X/\C}^H, {\rm Fil}_{\rm symbinf}^\bullet)$ give examples of filtered differential complexes. 

In Appendix \ref{sec-appFDC}, we show that the complex of K\"{a}hler differentials $(\Omega_X^\bullet,F)$ (considered as a complex on $Y$), with its stupid filtration, is another example. 

Moreover, if we form the complex (also considered in \cites{SVV,PSV,KovacsInjectivity})
\[ \Omega_{X,h}^\bullet = \left[\cH^0(\underline{\Omega}_X^0) \xrightarrow[]{d}\cH^0(\underline{\Omega}_X^1) \xrightarrow[]{d} \cdots\xrightarrow[]{d} \cH^0(\underline{\Omega}_X^{\dim X})\right],\]
via the connecting morphisms in the triangle
\[ F^{p+1}\underline{\Omega}_X^\bullet/F^{p+2}\underline{\Omega}_X^\bullet  \to F^p\underline{\Omega}_X^\bullet/F^{p+2}\underline{\Omega}_X^\bullet \to F^p\underline{\Omega}_X^\bullet/F^{p+1}\underline{\Omega}_X^\bullet \xrightarrow[]{+1},\]
then this gives another stupidly filtered differential complex.
\end{rmk}

\begin{lem} Let $(\cE^\bullet_{\rm der}, {\rm Fil}^H_{\bullet}) = {\rm DR}^{-1} \mathbb D(\widehat{\rm dR}_X, {\rm Fil}^H_\bullet)[-\dim Y]$. Then the filtered morphisms
\[ (\Omega_Y^\bullet,F_\bullet) \to (\widehat{\rm dR}_X, {\rm Fil}^H_\bullet) \to (\underline{\Omega}_X^\bullet,F_\bullet)\]
induce morphisms (after forgetting the filtrations)
\[ R\Gamma_X(\omega_Y) \to \cE^\bullet_{\rm der} \to R\Gamma_X(\omega_Y) \to \omega_Y.\]

These morphisms exhibit $R\Gamma_X(\omega_Y)$ as a direct summand of $\cE^\bullet_{\rm der}$ in the derived category of $\cD$-modules.
\end{lem}
\begin{proof} 
Note that ${\rm DR}^{-1} \mathbb D(\Omega_Y^\bullet,F) \simeq \omega_Y[\dim Y]$, as $Y$ is smooth. Thus, we get the composition
\[ R\Gamma_X(\omega_Y) \to \cE^\bullet_{\rm der} \to \omega_Y.\]

But then $\cE^\bullet_{\rm der}$ is set-theoretically supported on $X$, because each associated graded piece (of the exhaustive and bounded below filtration) ${\rm Gr}^F_p \cE^\bullet_{\rm der}$ is, so by \lemmaref{lem-adjunction} we get the factorization in the statement. The last claim follows by applying \lemmaref{lem-endConstant}.
\end{proof}

\begin{proof}[Proof of \theoremref{thm-Injectivity2}] We have the morphisms of filtered differential complexes
\[ (\Omega_{X/\C}^H,{\rm Fil}_{\rm inf}) \to (\Omega_{X/\C}^H,{\rm Fil}_{\rm symbinf}) \to (\Omega_X^\bullet,F) \to (\Omega_{X,h}^\bullet,F) \to (\underline{\Omega}_X^\bullet,F)\]
where $(\Omega_X^\bullet,F)$ and $(\Omega_{X,h}^\bullet,F)$ are given the stupid filtration.

The proof is organized by the following diagram.
\[ \begin{tikzcd} \cH^q F_p {\rm DR}^{-1} \mathbb D (\underline{\Omega}_X^\bullet,F_\bullet) \ar[r,"\cong"] \ar[d] & F_p \cH^q {\rm DR}^{-1} \mathbb D (\underline{\Omega}_X^\bullet,F_\bullet) \ar[r,hook] \ar[d] & \cH^q {\rm DR}^{-1} \mathbb D(\underline{\Omega}_X^\bullet) \ar[d] \\ 
\cH^q F_p {\rm DR}^{-1}\mathbb D(\Omega_{X,h}^\bullet,F_\bullet) \ar[r,two heads] \ar[d] & F_p \cH^q {\rm DR}^{-1}\mathbb D(\Omega_{X,h}^\bullet,F_\bullet) \ar[r,hook] \ar[d] & \cH^q {\rm DR}^{-1} \mathbb D(\Omega_{X,h}^\bullet) \ar[d] \\  \cH^q F_p {\rm DR}^{-1} \mathbb D(\Omega_X^\bullet,F_\bullet)  \ar[r,two heads] \ar[d] & F_p \cH^q {\rm DR}^{-1} \mathbb D(\Omega_X^\bullet,F_\bullet) \ar[r,hook] \ar[d] & \cH^q {\rm DR}^{-1} \mathbb D(\Omega_X^\bullet)  \ar[d] \\ 
\cH^q F_p {\rm DR}^{-1} \mathbb D(\Omega_{X/\C}^H,{\rm Fil}^{\rm symbinf}_\bullet) \ar[r,two heads] \ar[d] & F_p \cH^q {\rm DR}^{-1} \mathbb D(\Omega_{X/\C}^H,{\rm Fil}^{\rm symbinf}_\bullet)\ar[r,hook] \ar[d] & \cH^q {\rm DR}^{-1}  \mathbb D(\Omega_{X/\C}^H) \ar[d] \\ 
\cH^q F_p {\rm DR}^{-1} \mathbb D(\Omega_{X/\C}^H,{\rm Fil}^{\rm inf}_\bullet) \ar[r,two heads] \ar[d] & F_p \cH^q {\rm DR}^{-1} \mathbb D(\Omega_{X/\C}^H,{\rm Fil}^{\rm inf}_\bullet)\ar[r,hook]\ar[d] & \cH^q {\rm DR}^{-1}  \mathbb D(\Omega_{X/\C}^H) \ar[d] \\ 
\cH^q F_p {\rm DR}^{-1} \mathbb D(\widehat{\rm dR}_X,{\rm Fil}^{H}_\bullet) \ar[r,two heads] & F_p \cH^q {\rm DR}^{-1} \mathbb D(\widehat{\rm dR}_X,{\rm Fil}^{H}_\bullet) \ar[r,hook] & \cH^q {\rm DR}^{-1} \mathbb D(\widehat{\rm dR}_X) \end{tikzcd} .\]

Here the top left horizontal morphism is an isomorphism by strictness of filtered duality in Saito's theory. We noted above that the morphism from the top right to the bottom right is a split injection. Hence, we conclude that the morphism from the top left object to the bottom right object is injective, too. By commutativity of this diagram, we conclude that \emph{any} morphism in the diagram out of the top left object is injective. In particular, the morphism
\[ \cH^q F_p {\rm DR}^{-1} \mathbb D( \underline{\Omega}_X^\bullet,F) \to \cH^q F_p {\rm DR}^{-1} \mathbb D(\Omega^\bullet_{X,h},F)\]
is injective for all $p,q \in \Z$.

We now prove Statement \eqref{itm-InjectivePre} of \theoremref{thm-Injectivity2}; the remaining statements follow by the same argument.

Assuming $X$ is pre-$(k-1)$-Du Bois is equivalent to assuming that the morphism $(\Omega_{X,h}^\bullet,F) \to (\underline{\Omega}_X^\bullet,F)$ induces a quasi-isomorphism after applying ${\rm Gr}^F_{-p}(-)$ for all $p < k$. Dually, the morphism
\[ (\mathbb D(\underline{\Omega}_X^\bullet),F) \to ( \mathbb D( \Omega_{X,h}^\bullet),F)\]
induces a quasi-isomorphism on the $p$-th graded piece ${\rm Gr}^F_{p}(\mathbb D(\underline{\Omega}_X^\bullet)) \to {\rm Gr}^F_{p}(\mathbb D( \Omega_{X,h}^\bullet))$ and also a quasi-isomorphism $F_{p}\mathbb D(\underline{\Omega}_X^\bullet) \to F_{p}\mathbb D( \Omega_{X,h}^\bullet)$ by \corollaryref{cor-DoubleDual}. It therefore remains to prove that the induced morphism
\[ \cH^q {\rm Gr}^F_k (\mathbb D(\underline{\Omega}_X^\bullet)) \to \cH^q {\rm Gr}^F_k ( \mathbb D( \Omega_{X,h}^\bullet))\]
is injective for all $q\in \Z$.

The natural morphism
\[ \widetilde{\rm DR}{\rm DR}^{-1} \mathbb D(-) \to \mathbb D(-)\]
is a filtered quasi-isomorphism by~\eqref{eq-filtqisDRDRInv}; hence it suffices to prove that the morphism
\[ \cH^q {\rm Gr}^F_k \widetilde{\rm DR}{\rm DR}^{-1} (\mathbb D(\underline{\Omega}_X^\bullet)) \to \cH^q {\rm Gr}^F_k \widetilde{\rm DR}{\rm DR}^{-1} ( \mathbb D( \Omega_{X,h}^\bullet))\]
is injective for all $q\in \Z$. Once again using the filtered quasi-isomorphism, the assumption that $X$ is pre-$(k-1)$-Du Bois is equivalent to the property that
\[{\rm Gr}^F_p \widetilde{\rm DR}{\rm DR}^{-1} (\mathbb D(\underline{\Omega}_X^\bullet)) \to {\rm Gr}^F_p \widetilde{\rm DR}{\rm DR}^{-1} ( \mathbb D( \Omega_{X,h}^\bullet))\]
is a quasi-isomorphism for all $p < k$. By setting
\[ (\cM_2^\bullet,F) = {\rm DR}^{-1} (\mathbb D(\underline{\Omega}_X^\bullet),F),\]
\[ (\cM_1^\bullet,F) = {\rm DR}^{-1} (\mathbb D(\Omega_{X,h}^\bullet),F),\]
we see that we are exactly in the setting of \lemmaref{lem-Technical} above, which concludes the proof.
\end{proof}

In the case of $\Omega_{X/\C}^H$, with either its infinitesimal or symbolic infinitesimal filtration, the resulting complex of $\cD$-modules was indeed quasi-isomorphic to local cohomology, as was shown by direct computation. Moreover, the induced filtration on local cohomology is identified by \theoremref{thm-InterpretExt}. Bhatt shows that $\widehat{\rm dR}_X$, ignoring filtrations, is quasi-isomorphic to $\Omega_{X/\C}^H$, and so it is natural to wonder whether the same is true after taking the filtered dual. 
\begin{question} Is the natural map 
\[R\Gamma_X(\omega_Y) \to \cE^\bullet_{\rm der}\]
actually a quasi-isomorphism?
\end{question}

\begin{rmk} Under certain finiteness and holonomicity hypotheses, the answer to the above question would be positive by \cite{SaitoMHP}*{Prop. 2.4.12}. However, we cannot verify those conditions for the derived de Rham complex.
\end{rmk}

In any case, the morphism
\[ \cE^\bullet_{\rm der} \to R\Gamma_X(\omega_Y)\]
induces the filtration 
\[ E^{\rm der}_{p-\dim Y} \cH^q_X(\omega_Y) = {\rm im}\left(\cH^q {\rm Fil}^H_{p} \cE^\bullet_{\rm der} \to \cH^q_X(\omega_Y)\right).\] By construction, we have the chain of containments
\[ F_\bullet \cH^q_X(\omega_Y) \subseteq E_\bullet^{\rm symb} \cH^q_X(\omega_Y) \subseteq E_\bullet \cH^q_X(\omega_Y) \subseteq E_\bullet^{\rm der} \cH^q_X(\omega_Y). \]

\begin{question} Does the filtration $E_\bullet^{\rm der} \cH^q_X(\omega_Y)$ admit a more direct description?
\end{question}

\appendix 
\section{Construction of Filtered Differential Complexes} \label{sec-appFDC}
In the appendix, we provide a proof that certain filtered complexes are indeed filtered differential complexes in the sense of Saito.

We use the DGA (differential graded algebra) structure of the objects involved. The DGA structure on $\Omega_{X}^\bullet$ is classical, whereas that on $\Omega_{X,h}^\bullet$ follows immediately from Huber-J\"{o}rder's description of the $h$-differentials \cite{HuberJorder}*{Thm. 1}. We give a general construction which applies to either of these.

Let $(M,F,d,\wedge)$ be a filtered DGA on the smooth variety $Y$ with a filtered morphism $(\Omega_Y^\bullet,F,d,\wedge) \to (M,F,d,\wedge)$. Here we use increasing indices, so that, for example,
\[ F_{-j} \Omega_Y^p = \begin{cases} 0 & j > p \\ \Omega_Y^p & j \leq p \end{cases}.\]

By a filtered DGA, we mean $d\colon (M^i,F) \to (M^{i+1},F)$ is filtered and satisfies the graded Leibniz rule
\[ d(m_0 \wedge m_1) = d(m_0) \wedge m_1 + (-1)^{\deg(m_0)} m_0 \wedge d(m_1),\]
and
\[ (F_j M) \wedge (F_\ell M) \subseteq F_{j+\ell} M.\]

In particular, for all $\alpha \in \Omega_Y^1$ and $m\in F_j M$, we have
\begin{equation} \label{eq-wedgeDropHodge} m \wedge [\alpha] \in F_{j-1} M.\end{equation}

\begin{eg} If $M$ is any DGA, then $(M,F)$ where $F$ is the stupid filtration, so that $F_{-j}M^p = \begin{cases} 0 & j > p \\ M^p & j \leq p\end{cases}$, is a filtered DGA.
\end{eg}

\begin{prop} \label{prop-Construction2} If $(\Omega_Y^\bullet,F,\wedge,d) \to (M^\bullet,F,\wedge,d)$ is a morphism of filtered DGAs, then the complex
\[ \cdots \xrightarrow[]{d} (M^{-1},F) \xrightarrow[]{d} (M^0,F) \xrightarrow[]{d} (M^1,F) \xrightarrow[]{d} \cdots\]
is a filtered differential complex.
\end{prop}
\begin{proof}
Indeed, we have the natural morphism $\Omega_Y^1 \to M^1$ and so, for any $p \in \Z$, we have the composition
\[ M^p \otimes_{\cO_Y} \Omega_Y^1 \to M^p \otimes_{\cO_Y} M^1 \xrightarrow[]{\wedge} M^{p+1}\]
which gives a natural morphism
\[ \widetilde{\nabla} \colon M^p \otimes_{\cO_Y} \cD_Y \to M^{p+1} \otimes_{\cO_Y} \cD_Y.\]

In local coordinates $y_1,\dots, y_n$ on $Y$, this map is given by
\[ \widetilde{\nabla}(m \otimes P) = d(m) \otimes P + \sum_{i=1}^n (-1)^p m \wedge [dy_i] \otimes (\de_{y_i} P),\]
where $[dy_i]$ is the image of $dy_i$ in $M^1$. It is clear that this is right $\cD_Y$-linear and $\widetilde{\nabla}^2 = 0$. Keeping the filtration in mind, we get a complex of filtered right $\cD_Y$-modules
\[ \cdots \to (M^{-1},F) \otimes_{\cO_Y} (\cD_Y,F) \xrightarrow[]{\widetilde{\nabla}}  (M^0,F) \otimes_{\cO_Y} (\cD_Y,F) \xrightarrow[]{\widetilde{\nabla}} (M^1,F) \otimes_{\cO_Y} (\cD_Y,F)\xrightarrow[]{\widetilde{\nabla}} \cdots. \]

Indeed, we can easily see that it is filtered by \eqref{eq-wedgeDropHodge}. Using the fact that $\partial_{y_i}(1) = 0$ for all $1\leq i\leq \dim Y$, we see that after applying the functor $\widetilde{\rm DR}(-)$ we recover the complex $((M^\bullet,F),d)$, proving the claim.
\end{proof}

\bibliography{bibliography}

\begin{bibdiv}
\begin{biblist}

\bib{BdJ}{article}{
      author={Bhatt, Bhargav},
      author={de~Jong, Aise~Johan},
       title={Crystalline cohomology and de {R}ham cohomology},
        date={2011},
     journal={preprint arXiv: 1110.5001},
}

\bib{BhattDDR}{article}{
      author={Bhatt, Bhargav},
       title={Completions and derived de {R}ham cohomology},
        date={2012},
     journal={preprint arXiv: 1207.6193},
}

\bib{CDOHigher}{article}{
      author={Chen, Qianyu},
      author={Dirks, Bradley},
      author={Olano, Sebasti\'{a}n},
       title={Higher singularities through local cohomology},
        date={2026},
}

\bib{DeligneDimca}{article}{
      author={Deligne, Pierre},
      author={Dimca, Alexandru},
       title={Filtrations de {H}odge et par l'ordre du p{$\^{o}$}le pour les
  hypersurfaces singuli\`{e}res},
        date={1990},
     journal={Ann. Sci. {\'{E}}c. Norm. Sup\'{e}r.},
      volume={23},
      number={4},
       pages={645\ndash 656},
}

\bib{Eisenbud}{book}{
      author={Eisenbud, David},
       title={Commutative {A}lgebra with a {V}iew {To}ward {Algebraic}
  {Geometry}},
     edition={1},
      series={Graduate Texts in Mathematics},
   publisher={Springer New York, NY},
        date={1995},
}

\bib{EMS}{article}{
      author={Eisenbud, David},
      author={{Musta\c{t}\u{a}}, Mircea},
      author={Stillman, Mike},
       title={Cohomology on {T}oric {V}arieties and {L}ocal {C}ohomology with
  {M}onomial {S}upports},
        date={2000},
     journal={J. Symbolic Comput.},
      volume={29},
       pages={583\ndash 600},
}

\bib{Golovin}{article}{
      author={Golovin, V.D.},
       title={The global dimension of the sheaf of germs of holomorphic
  functions},
        date={1975},
     journal={Dokl. Akad. Nauk SSSR},
      volume={223},
      number={2},
       pages={273\ndash 275},
}

\bib{GriffithsPole}{article}{
      author={Griffiths, Philip},
       title={On the {P}eriods of {C}ertain {R}ational {I}ntegrals: {I}},
        date={1969},
     journal={Ann. of Math.},
      volume={90},
      number={3},
       pages={460\ndash 495},
}

\bib{HartshorneDR}{article}{
      author={Hartshorne, Robin},
       title={On the {D}e {R}ham cohomology of algebraic varieties},
        date={1975},
        ISSN={0073-8301,1618-1913},
     journal={Inst. Hautes \'Etudes Sci. Publ. Math.},
      number={45},
       pages={5\ndash 99},
  url={http://www.numdam.org.myaccess.library.utoronto.ca/item?id=PMIHES_1975__45__5_0},
      review={\MR{432647}},
}

\bib{HH}{article}{
      author={Hochster, Melvin},
      author={Huneke, Craig},
       title={Comparison of symbolic and ordinary powers of ideals},
        date={2002},
     journal={Invent. Math.},
      volume={147},
       pages={349\ndash 369},
}

\bib{HuberJorder}{article}{
      author={Huber, Annette},
      author={J\"{o}rder, Clemens},
       title={Differential forms in the h-topology},
        date={2014},
     journal={Algebr. Geom.},
      volume={1},
      number={4},
       pages={449\ndash 478},
}

\bib{HTT}{book}{
      author={Hotta, Ryoshi},
      author={Takeuchi, Kiyoshi},
      author={Tanisaki, Toshiyuki},
       title={D-{Modules}, {Perverse} {Sheaves}, and {Representation}
  {Theory}},
   publisher={{Birkha\"{u}ser} Boston},
        date={2008},
}

\bib{DBSaito}{article}{
      author={Jung, Seung-Jo},
      author={Kim, In-Kyun},
      author={Saito, Morihiko},
      author={Yoon, Youngho},
       title={Higher {Du} {Bois} singularities of hypersurfaces},
        date={2022},
     journal={Proc. Lond. Math. Soc.},
      volume={125},
      number={3},
       pages={543\ndash 567},
}

\bib{KovacsInjectivity}{article}{
      author={Kov\'{a}cs, S\'{a}ndor},
       title={Complexes of differential forms and singularities: {T}he
  injectivity theorem},
        date={2025},
     journal={preprint arXiv:2505.09912},
}

\bib{DBDeform}{incollection}{
      author={Kov\'{a}cs, S\'{a}ndor},
      author={Schwede, Karl},
       title={Du {B}ois singularities deform},
        date={2016},
   booktitle={Minimal models and extremal rays (kyoto, 2011)},
      volume={70},
   publisher={Math. Soc. Japan},
       pages={49\ndash 65},
}

\bib{MP3}{article}{
      author={{Musta\c{t}\u{a}}, Mircea},
      author={Popa, Mihnea},
       title={Hodge filtration on local cohomology, {Du} {Bois} complex, and
  local cohomological dimension},
        date={2022},
     journal={Forum Math. Pi},
      volume={10},
      number={e22},
       pages={58 pp.},
}

\bib{MPDB}{article}{
      author={{Musta\c{t}\u{a}}, Mircea},
      author={Popa, Mihnea},
       title={On {$k$-rational} and {$k$-Du Bois} local complete
  intersections},
        date={2025},
     journal={Algebr. Geom.},
      volume={12},
      number={2},
       pages={237\ndash 261},
}

\bib{PSMHS}{book}{
      author={Peters, Chris A.~M.},
      author={Steenbrink, Joseph H.~M.},
       title={Mixed {Hodge} {Structures}},
     edition={1},
      series={H.M. Steenbrink Ergebnisse der Mathematik und ihrer Grenzgebiete.
  3. Folge},
   publisher={Springer Berlin, Heidelberg},
        date={2008},
}

\bib{PSV}{article}{
      author={Popa, Mihnea},
      author={Shen, Wanchun},
      author={Vo, Anh~Duc},
       title={Injectivity and vanishing for the {D}u {B}ois {C}omplexes of
  {I}solated {S}ingularities},
        date={2026},
     journal={Algebr. Number Theory},
      volume={20},
       pages={1235\ndash 1262},
}

\bib{SaitoMHC}{article}{
      author={Saito, Morihiko},
       title={Mixed {Hodge} complexes on algebraic varieties},
        date={2000},
     journal={Math. Ann.},
      volume={316},
       pages={283\ndash 331},
}

\bib{SaitoMHP}{article}{
      author={Saito, Morihiko},
       title={Modules de {Hodge} {Polarisables}},
        date={1988},
     journal={Publ. Res. Inst. Math. Sci.},
      volume={24},
      number={6},
       pages={849\ndash 995},
}

\bib{SaitoDifferentialComplexes}{article}{
      author={Saito, Morihiko},
       title={Induced {$\mathscr{D}$}-modules and differential complexes},
        date={1989},
     journal={Bull. Soc. Math. France},
      volume={117},
      number={3},
       pages={361\ndash 387},
}

\bib{SaitoMHM}{article}{
      author={Saito, Morihiko},
       title={Mixed {Hodge} {Modules}},
        date={1990},
     journal={Publ. Res. Inst. Math. Sci.},
      volume={26},
      number={2},
       pages={221\ndash 333},
}

\bib{SVV}{article}{
      author={Shen, Wanchun},
      author={Venkatesh, Sridhar},
      author={Vo, Anh~Duc},
       title={On {$k$}-{Du} {Bois} and {$k$}-rational singularities},
        date={2023},
     journal={preprint arXiv:2306.03977, to appear in Ann. Inst. Fourier.},
}

\end{biblist}
\end{bibdiv}

\end{document}